\documentclass[a4paper,11pt]{article}
\pdfoutput=1
\usepackage{hyperref}
\hypersetup{hypertexnames = false, bookmarksdepth = 2, bookmarksopen = true, colorlinks, linkcolor = black, citecolor = black, urlcolor = black, pdfstartview={XYZ null null 1}}

\usepackage{amsfonts}
\usepackage[fleqn, leqno]{amsmath}
\usepackage{amsthm}

\usepackage{zref-clever}
\zcsetup{cap}
\zcsetup{noabbrev}
\newcommand{\cref}[1]{\zcref{#1}}
\newcommand{\Cref}[1]{\zcref[S]{#1}}

\usepackage{booktabs}
\usepackage{colonequals}
\usepackage{enumitem}
\usepackage{fullpage}
\usepackage{mathtools}
\usepackage{parskip}
\usepackage{thmtools}
\usepackage{tikz-cd}
\usepackage{xparse}
\usepackage{xpatch}
\usepackage{xspace}

\usepackage[utf8]{inputenc}
\usepackage[T1]{fontenc}
\usepackage{libertine}
\usepackage[libertine]{newtxmath}
\usepackage[scaled=0.83]{beramono}
\usepackage{eucal}
\usepackage{microtype}
\usepackage{gitinfo2}

\usetikzlibrary{arrows.meta}
\usetikzlibrary{calc}

\usepackage[backend=biber, maxbibnames=99, datamodel=mrnumber, sortcites]{biblatex}

\DeclareFieldFormat{mrnumber}{%
  MR\addcolon\space
  \ifhyperref
    {\href{http://www.ams.org/mathscinet-getitem?mr=#1}{\nolinkurl{#1}}}
    {\nolinkurl{#1}}}

\renewbibmacro*{doi+eprint+url}{%
  \iftoggle{bbx:doi}
    {\printfield{doi}}
    {}%
  \newunit\newblock
  \printfield{mrnumber}%
  \newunit\newblock
  \iftoggle{bbx:eprint}
    {\usebibmacro{eprint}}
    {}%
  \newunit\newblock
  \iftoggle{bbx:url}
    {\usebibmacro{url+urldate}}
    {}}

\DeclareSourcemap{
  \maps[datatype=bibtex]{
    \map{
      \step[fieldset=issn, null]
    }
  }
}
\xpretobibmacro{bib+doi+url}
  {\iffieldundef{doi}
     {}
     {\clearfield{url}}}
  {}{}

\xpretobibmacro{cite+doi+url}
  {\iffieldundef{doi}
     {}
     {\clearfield{url}}}
  {}{}

\AtEveryCite{\upshape}

\declaretheoremstyle[spaceabove = 3pt, spacebelow = 0pt, bodyfont = \itshape]{theorem}
\declaretheoremstyle[spaceabove = 3pt, spacebelow = 0pt]{remark}

\declaretheorem[style=theorem]{theorem}
\declaretheorem[style=theorem, sibling=theorem]{corollary}
\declaretheorem[style=theorem, sibling=theorem]{lemma}
\declaretheorem[style=theorem, sibling=theorem]{proposition}

\declaretheorem[style=remark, sibling=theorem]{definition}

\declaretheorem[style=remark, sibling=theorem]{notation}
\declaretheorem[style=remark, sibling=theorem]{remark}

\declaretheorem[style=theorem, numberwithin=section, title=Theorem]{alphatheorem}

\zcRefTypeSetup{alphatheorem}{Name-sg=Theorem,name-sg=theorem,Name-pl=Theorems,name-pl=theorems}
\zcRefTypeSetup{alphaconjecture}{Name-sg=Conjecture,name-sg=conjecture,Name-pl=Conjectures,name-pl=conjectures}
\zcRefTypeSetup{alphacorollary}{Name-sg=Corollary,name-sg=corollary,Name-pl=Corollaries,name-pl=corollaries}
\zcRefTypeSetup{alphaproposition}{Name-sg=Proposition,name-sg=proposition,Name-pl=Propositions,name-pl=propositions}

\zcRefTypeSetup{enumi}{refbounds={,\textup{(},\textup{)},}}

\setlist[enumerate]{font=\normalfont}

\mathchardef\mhyphen="2D

\newcommand\vect[1]{\ensuremath{\boldsymbol #1}}

\newcommand{\quivertools}{{\tt{QuiverTools}}\xspace}

\newcommand\QM[1]{%
  \ifcase#1
    \mathbf{?}%
    \or \texorpdfstring{\ensuremath{\mathrm{M}(1,2^4;4)}}{M(1,2,2,2,2;4)}%
    \or \texorpdfstring{\ensuremath{\mathrm{M}(1^5,2;3)}}{M(1,1,1,1,1,2;3)}%
    \or \texorpdfstring{\ensuremath{\mathrm{M}(1^4,2^2;3)}}{M(1,1,1,1,2,2;3)}%
    \or \texorpdfstring{\ensuremath{\mathrm{M}(1^7;2)}}{M(1,1,1,1,1,1,1;2)}%
    \else
    \mathbf{?}%
  \fi
}

\newcommand{\framed}{\mhyphen\mathrm{fr}}
\newcommand{\stable}{\mhyphen\mathrm{st}}
\newcommand{\stableorsemistable}{\mhyphen\mathrm{(s)st}}
\newcommand{\semistable}{\mhyphen\mathrm{sst}}

\DeclareDocumentCommand\modulispace{om}{\IfNoValueTF{#1}{\mathrm{M}{(#2)}}{\mathrm{M}^{#1}(#2)}}
\DeclareDocumentCommand\repspace{om}{\IfNoValueTF{#1}{\mathrm{R}{(#2)}}{\mathrm{R}^{#1}(#2)}}

\newcommand\can{\ensuremath{\mathrm{can}}}
\newcommand\nilpotent{\ensuremath{\mathrm{nilp}}}

\DeclareMathOperator\Aut{Aut}
\DeclareMathOperator\Bl{Bl}
\DeclareMathOperator\dimvect{\underline{dim}}

\DeclareMathOperator\GL{GL}
\DeclareMathOperator\Gm{\mathbb{G}_{\mathrm{m}}}
\DeclareMathOperator\Gr{Gr}
\DeclareMathOperator\Hom{Hom}

\DeclareMathOperator\PGL{PGL}

\DeclareMathOperator\Stab{Stab}
\DeclareMathOperator\supp{supp}
\DeclareMathOperator\Sym{Sym}

\newcommand{\msubspacequiverdims}[4][0.9]{%
  \begin{tikzpicture}%
    \pgfmathsetmacro{\msqR}{#1}%
    \pgfmathsetmacro{\msqLR}{#1 + 0.40}%
    \node (msqC) at (0,0) {};%
    \foreach \k in {1,...,#2}{%
      \pgfmathsetmacro{\msqA}{90 - ((\k - 1)*360.0/#2)}%
      \node (msqV\k) at (\msqA:\msqR cm) {};%
      \draw[->] (msqV\k) edge (msqC);%
    }%
    \foreach \msqLbl [count=\k] in {#3}{%
      \pgfmathsetmacro{\msqA}{90 - ((\k - 1)*360.0/#2)}%
      \draw (msqV\k) circle (2pt);%
      \node at (\msqA:\msqLR cm) {$\msqLbl$};%
    }%
    \draw (msqC) circle (2pt) node [right=0.5em, inner sep=2pt] {$#4$};%
  \end{tikzpicture}%
}

\usepackage{xpatch}

\makeatletter
\patchcmd\blx@bblinput{\blx@blxinit}
                      {\blx@blxinit
                      }{}{\fail}
\makeatother

\title{Fano 4-fold quiver moduli from subspace quivers}
\author{Pieter Belmans \and Markus Reineke}

\makeatletter%
\begin{document}
\maketitle

\begin{abstract}
  We classify the moduli spaces of representations of subspace quivers
  which are Fano fourfolds, under a natural assumption on the dimension vector.
  These moduli spaces can also be described as
  GIT quotients of products of Grassmannians
  by the diagonal action of a projective linear group,
  and there are exactly four of them.
  They are rational, of pure Hodge--Tate type, infinitesimally rigid,
  and have finite automorphism groups,
  with Picard ranks~5,~6,~6 and~7,
  making them interesting examples in
  the classification of Fano fourfolds of large Picard rank,
  as they are not toric or products.
  Two are known varieties: Manivel's Segre cousin of the Segre cubic 3-fold,
  and the Fano model of the blowup of~$\mathbb{P}^4$ in six points.
  The other two appear to be new:
  one is an involution surface bundle over~$\mathbb{P}^2$,
  and the other is a ``Segre cousin once-removed'',
  whose geometry closely parallels that of the Segre cousin.
  Using techniques from quiver moduli, which we survey,
  we describe the geometry of all four fourfolds in detail.
\end{abstract}

{\setcounter{tocdepth}{1}%
\makeatletter
\patchcmd{\l@section}{\addvspace{1.0em \@plus\p@}}{\addvspace{0.2em \@plus\p@}}{}{}%
\makeatother
\tableofcontents}

\section{Introduction}
In \cite{MR4739832}
Manivel studies a four-dimensional cousin of the famous Segre cubic.
It arose as an interesting example in the ongoing classification of Fano 4-folds
using zero loci of homogeneous vector bundles on partial flag varieties \cite{fanofourfolds}:
it is the Fano 4-fold with largest Picard number (namely~6) in this database;
moreover it is infinitesimally rigid and has finite automorphism group.

In \cite[\S6]{MR4622141}
Casagrande surveys known Fano 4-folds of large Picard number (meaning at least~6),
observing that there are very few known examples which are not
toric varieties or products of lower-dimensional Fano varieties,
and that one of those interesting examples is Manivel's Segre cousin.
Another one, of Picard rank~7,
is the Fano model of~$\Bl_6\mathbb{P}^4$ \cite[Example~25]{MR4622141}.

In this article we use quiver moduli for subspace quivers
to exhibit more Fano 4-folds with the same remarkable properties as the Segre cousin.
In particular,
we realize both the Segre cousin
and the Fano model of~$\Bl_6\mathbb{P}^4$ using quiver moduli.
But we also discover a Fano 4-fold which one could call
a \emph{Segre cousin once-removed}:
it shares many, but not all, numerical invariants with the Segre cousin,
and the two varieties behave very similarly.
Namely,
like the Segre cousin
it admits a birational semismall contraction to a smooth fourfold,
with~$\Gr(2,4)$ replaced by~$\mathbb{P}^2\times\mathbb{P}^2$,
and the discriminant configuration of five planes meeting in ten points
is replaced by four quadrics meeting in six points.
We moreover find a Fano 4-fold of Picard rank~5,
with a rich fibration structure.

Our main result is the following.
\begin{alphatheorem}
  \label{theorem:main}
  There are Fano fourfolds defined as GIT quotients of products of Grassmannians
  \begin{equation}
    \begin{aligned}
      \QM{1}&=(\Gr(1,4)\times\Gr(2,4)^4)_{\rm st}/\!/\PGL_4,\\
      \QM{2}&=(\Gr(1,3)^5\times\Gr(2,3))_{\rm st}/\!/\PGL_3,\\
      \QM{3}&=(\Gr(1,3)^4\times\Gr(2,3)^2)_{\rm st}/\!/\PGL_3,\\
      \QM{4}&=(\Gr(1,2)^7)_{\rm st}/\!/\PGL_2
    \end{aligned}
  \end{equation}
  with stability as in \eqref{equation:stability},
  which are rational, of pure Hodge--Tate type, infinitesimally rigid,
  and admit no vector fields.
  Alternatively,
  they are constructed as quiver moduli
  using the subspace quivers and dimension vectors in \cref{table:quivers},
  for the canonical stability parameter.
  Some of their numerical invariants, including the Chern numbers, are surveyed in \cref{table:overview}.

  Moreover:
  \begin{enumerate}
    \item The variety \QM{1} admits a surjective morphism to~$\mathbb{P}^2$,
      for which it is an involution surface bundle of mild degeneration in the sense of \cite{MR4159824}.
    \item The variety \QM{2} admits a semismall morphism to~$\Gr(2,4)$.
      In fact, it is isomorphic to the Segre cousin of \cite{MR4739832}.
    \item The variety \QM{3} admits a semismall morphism to~$\mathbb{P}^2\times\mathbb{P}^2$.
    \item The variety \QM{4} admits a (non-flat) morphism to the Segre cubic, with general fibre~$\mathbb{P}^1$.
      In fact, it is isomorphic to the Fano model of~$\Bl_6\mathbb{P}^4$.
  \end{enumerate}
\end{alphatheorem}

The computation of the invariants in \cref{table:overview}
uses \cite{quivertools},
and is given in \cite{FanoFourfoldSubspaceQuiverModuli.jl}.

\begin{table}[t!]
  \centering
  \begin{tabular}{cccccccc}
    \toprule
    variety & Betti numbers & $\mathrm{h}^0(X,\omega_X^\vee)$ & $\mathrm{c}_1^4$ & $\mathrm{c}_2\mathrm{c}_1^2$ & $\mathrm{c}_3\mathrm{c}_1$ & $\mathrm{c}_2^2$ & $\mathrm{c}_4$ \\
    \midrule
    $\QM{1}$ & 1, 5, 13, 5, 1 & 42 & 181 & 130 & 58 & 116 & 25 \\
    $\QM{2}$ & 1, 6, 17, 6, 1 & 40 & 172 & 124 & 58 & 123 & 31 \\
    $\QM{3}$ & 1, 6, 17, 6, 1 & 42 & 182 & 128 & 58 & 121 & 31 \\
    $\QM{4}$ & 1, 7, 22, 7, 1 & 36 & 154 & 112 & 56 & 136 & 38 \\
    \bottomrule
  \end{tabular}
  \caption{Properties of the four Fano quiver moduli from subspace quivers; $\mathrm{c}_i=\mathrm{c}_i(\mathrm{T}_X)$}
  \label{table:overview}
\end{table}

\begin{table}[t]
  \centering
  \begin{tabular}{cccc}
    \toprule
    $\QM{1}$ & $\QM{2}$ & $\QM{3}$ & $\QM{4}$ \\
    \midrule
    \raisebox{-1.2cm}{\msubspacequiverdims[0.7]{5}{1,2,2,2,2}{4}} &
    \raisebox{-1.2cm}{\msubspacequiverdims[0.7]{6}{1,1,1,1,1,2}{3}} &
    \raisebox{-1.2cm}{\msubspacequiverdims[0.7]{6}{1,1,1,1,2,2}{3}} &
    \raisebox{-1.2cm}{\msubspacequiverdims[0.7]{7}{1,1,1,1,1,1,1}{2}} \\
    \bottomrule
  \end{tabular}
  \caption{Subspace quivers and dimension vectors for the varieties in \cref{theorem:main}}
  \label{table:quivers}
\end{table}

We obtain these results by using that these Fano fourfolds are quiver moduli.
The techniques we need are scattered over the quiver moduli literature;
we survey them in \cref{section:techniques},
also to make them more widely known.

We point out that no description of~$\QM{1}$ and~$\QM{4}$
as the zero locus of a general section of a homogeneous vector bundle
on a product of Grassmannians is currently known,
unlike for~$\QM{2}$ and~$\QM{3}$,
see \cref{subsection:segre-cousin} and \cref{remark:enrico}.

\paragraph{Structure of the paper}
In \cref{section:quiver-moduli} we recall the basics on quiver moduli,
and in \cref{section:techniques} we survey various techniques to study quiver moduli,
which we will use in the rest of the paper.
In \cref{section:classification} we classify all Fano fourfolds arising as moduli spaces of stable representations of a subspace quiver,
subject to a natural condition, see \cref{theorem:classification}.
The remaining sections each provide a geometric description of one of the four varieties;
the descriptions from \cref{theorem:main}
are respectively given in
\cref{theorem:involution-surface-bundle},
\cref{theorem:QM2-grassmannian},
\cref{theorem:QM3-fibration},
and \cref{theorem:QM4-segre-cubic}.
The identification of~$\QM{2}$ with the Segre cousin is \cref{theorem:segre-cousin},
whilst the identification of~$\QM{4}$ with the Fano model of~$\Bl_6\mathbb{P}^4$
is due to Bolognesi--Massarenti \cite[\S2.35]{MR4243655},
see \cref{proposition:QM4-fano-model}.

\paragraph{Acknowledgements}
We would like to thank Gianni Petrella
for implementing performance improvements to \quivertools,
which made the computations in the Chow ring of~$\QM{1}$ possible.
We want to thank Christian Lehn for interesting discussions,
which made \cref{subsection:segre-cousin} possible,
and we want to thank Enrico Fatighenti for the suggestion in \cref{remark:enrico}.
LLMs were used to pin down the configuration of points in \cref{subsection:segre-cousin}.
P.B.~was partially supported by NWO (\href{https://doi.org/10.61686/RZKLF82806}{\texttt{doi:10.61686/RZKLF82806}}).
This work was done during the residency of M.R.~at Utrecht University
as the 2025--2026 Springer Visiting Professor,
and we gratefully acknowledge the support of the Utrecht Geometry Centre.

\section{Quiver moduli}
\label{section:quiver-moduli}
In this section we will recall some basics on quiver moduli,
see \cite{MR2484736,MR3727119} for some background.

\subsection{Quiver notation}
Let~$Q$ be a quiver, that is,
a finite oriented multigraph with set of vertices~$Q_0$ and set of arrows~$Q_1$,
where arrows are written~$\alpha\colon i\rightarrow j$,
and we write~$\operatorname{s}(\alpha)=i$ for the source
and~$\operatorname{t}(\alpha)=j$ for the target of an arrow.
Let~$\mathbb{Q}Q_0$ be the~$\mathbb{Q}$-vector space with basis elements~$\vect{i}$ for~$i\in Q_0$,
let~$\mathbb{Q}Q_0^\vee$ be its linear dual,
let~$\mathbb{Z}Q_0$ be the~$\mathbb{Z}$-lattice generated by the~$\vect{i}$,
and let~$\mathbb{N}Q_0$ be the cone generated by the~$\vect{i}$.
Elements of~$\mathbb{Q}Q_0$ are written~$\vect{d}=\sum_{i\in Q_0}d_i\vect{i}$ for~$d_i\in\mathbb{Q}$.
We define the support of an element~$\vect{d}\in\mathbb{N}Q_0$ as~$\supp(\vect{d})=\{i\in Q_0\mid d_i>0\}$.

On~$\mathbb{Q}Q_0$, we consider the (in general non-symmetric) Euler form~$\langle\_,\_\rangle$ given by
\begin{equation}
  \langle\vect{d},\vect{e}\rangle=\sum_{i\in Q_0}d_ie_i-\sum_{\alpha\colon i\rightarrow j}d_ie_j,
\end{equation}
and we denote its symmetrization by~$(\vect{d},\vect{e})=\langle\vect{d},\vect{e}\rangle+\langle\vect{e},\vect{d}\rangle$.
We also consider the standard bilinear form~$\vect{d}\cdot\vect{e}=\sum_{i\in Q_0}d_ie_i$.

We define the \emph{fundamental domain}~$\mathrm{F}_Q\subset\mathbb{N}Q_0$ of~$Q$
as the set of all~$\vect{d}\not=0$ with connected support~$\supp(\vect{d})$
such that~$(\vect{d},\vect{i})\leq 0$ for all~$i\in Q_0$.
This is empty if~$Q$ is of Dynkin type,
a single ray~$\mathbb{N}_{>0}\vect{d}_0$ if~$Q$ is of extended Dynkin type,
and a convex cone (without vertex) in~$\mathbb{N}Q_0$ otherwise.

We define reflections~$\mathrm{s}_i$ on~$\mathbb{Z}Q_0$ for~$i\in Q_0$ by
\begin{equation}
  \mathrm{s}_i(\vect{d})=\vect{d}-(\vect{d},\vect{i})\vect{i}.
\end{equation}

Throughout this paper, we will work over the base field~$\mathbb{C}$ of complex numbers.
We consider finite-dimensional~$\mathbb{C}$-representations~$V$ of the quiver~$Q$
given by~$\mathbb{C}$-vector spaces~$V_i$ for~$i\in Q_0$,
and~$\mathbb{C}$-linear maps~$V(\alpha)\colon V_i\rightarrow V_j$ for~$(\alpha\colon i\rightarrow j)\in Q_1$.
The dimension vector~$\dimvect(V)\in\mathbb{N}Q_0$ is given as~$\dimvect(V)=\sum_{i\in Q_0}\dim_\mathbb{C}V_i\vect{i}$.

\subsection{Quiver moduli}
For a given dimension vector~$\vect{d}\in\mathbb{N}Q_0$
and fixed vector spaces~$V_i$ of dimension~$d_i$ for~$i\in Q_0$,
we consider the space
\begin{equation}
  \repspace{Q,\vect{d}}=\bigoplus_{\alpha\colon i\rightarrow j}\Hom_\mathbb{C}(V_i,V_j)
\end{equation}
parametrizing representations of~$Q$ on the~$V_i$, on which the group
\begin{equation}
  \GL_{\vect{d}}=\prod_{i\in Q_0}\GL(V_i)
\end{equation}
acts by change of basis
\begin{equation}
  (g_i)_i\cdot(V_\alpha)_\alpha=(g_jV_\alpha g_i^{-1})_{\alpha:i\rightarrow j},
\end{equation}
such that the orbits under this action naturally correspond to
isomorphism classes of representations of~$Q$ of dimension vector~$\vect{d}$.

We denote by~$\Stab(\vect{d})\subset\mathbb{Q}Q_0^\vee$
the hyperplane of all~$\Theta$ annihilating~$\vect{d}$,
and refer to it as the space of stability parameters for~$\vect{d}$.
Given such a stability parameter~$\Theta$, we call a representation~$V$ as before~$\Theta$-semistable
(resp.~$\Theta$-stable)
if~$\Theta(\dimvect(U))\leq 0$
(resp.~$<0$)
for all non-zero proper subrepresentations~$U$ of~$V$.
We call~$V$~$\Theta$-polystable if it is isomorphic to a direct sum~$V=\bigoplus_{k=1}^s{U_k}^{m_k}$
of~$\Theta$-stable representations~$U_k$ (in particular,~$\Theta(\dimvect(U_k))=0$).

The~$\GL_{\vect{d}}$-orbit of a representation~$V$
in the open locus~$\repspace[\Theta\semistable]{Q,\vect{d}}\subset \repspace{Q,\vect{d}}$
of~$\Theta$-semistable representations is closed
if and only if~$V$ is~$\Theta$-polystable.
Thus, by Geometric Invariant Theory,
there exists a complex algebraic variety~$\modulispace[\Theta\semistable]{Q,\vect{d}}$
parametrizing isomorphism classes of~$\Theta$-polystable representations of~$Q$ of dimension vector~$\vect{d}$.
It contains a smooth irreducible open subvariety~$\modulispace[\Theta\stable]{Q,\vect{d}}$
parametrizing the~$\GL_{\vect{d}}$-orbits
in the locus~$\repspace[\Theta\stable]{Q,\vect{d}}\subset\repspace[\Theta\semistable]{Q,\vect{d}}$ of~$\Theta$-stable representations,
which has dimension~$1-\langle\vect{d},\vect{d}\rangle$ if non-empty.

For two dimension vectors~$\vect{e}\leq\vect{d}$,
we write~$\vect{e}\hookrightarrow\vect{d}$ if every representation of dimension vector~$\vect{d}$
contains a subrepresentation of dimension vector~$\vect{e}$,
following Schofield \cite{MR1162487}.
If this is not the case, the set of representations~$V\in\repspace{Q,\vect{d}}$
admitting a subrepresentation of dimension vector~$\vect{e}$ is a proper closed subset.
Consequently, we find that~$\repspace[\Theta\semistable]{Q,\vect{d}}$,
and thus,~$\modulispace[\Theta\semistable]{Q,\vect{d}}$,
is non-empty if and only if~$\Theta(\vect{e})\leq 0$ for all~$\vect{e}\hookrightarrow\vect{d}$.
The latter can be decided recursively,
since~$\vect{e}\hookrightarrow\vect{d}$
if and only if~$\langle\vect{e}',\vect{d}-\vect{e}\rangle\geq 0$ for all~$\vect{e}'\hookrightarrow\vect{e}$.
Dually, we define~$\vect{d}\twoheadrightarrow\vect{f}$
if every representation of dimension vector~$\vect{d}$ admits a factor representation of dimension vector~$\vect{f}$.
For~$\vect{d}=\vect{e}+\vect{f}$, we have~$\vect{d}\twoheadrightarrow\vect{f}$ if and only if~$\vect{e}\hookrightarrow\vect{d}$.

The moduli space~$\modulispace[\Theta\semistable]{Q,\vect{d}}$ equals the Proj of
the~$\mathbb{N}$-graded ring~$\mathbb{C}[\repspace{Q,\vect{d}}]^{\GL_{\vect{d}}}_\Theta$
of polynomial functions~$f\in\mathbb{C}[\repspace{Q,\vect{d}}]$ on~$\repspace{Q,\vect{d}}$
which are~$\GL_{\vect{d}}$-semi-invariant of weight~$C\cdot\Theta$,
for~$C\in\mathbb{N}$, in the following sense:
\begin{equation}
  f((g_i)\cdot V)=\left( \prod_{i\in Q_0}\det(g_i)^{-\Theta(\vect{i})} \right)^C\cdot f(V).
\end{equation}

In particular, for~$\Theta=0$ we find that~$\modulispace[0\semistable]{Q,\vect{d}}$
parametrizes semisimple representations of~$Q$ of dimension vector~$\vect{d}$.
It is a cone induced by the dilation action on~$\repspace{Q,\vect{d}}$ with vertex~$0$.
In case~$Q$ is an acyclic quiver (that is, it does not admit oriented cycles),~$\modulispace[0\semistable]{Q,\vect{d}}$
reduces to a point~$\{0\}$.
The coordinate ring of~$\modulispace[0\semistable]{Q,\vect{d}}$
is the ring~$\mathbb{C}[\repspace{Q,\vect{d}}]^{\GL_{\vect{d}}}$
of~$\GL_{\vect{d}}$-invariant polynomial functions on~$\repspace{Q,\vect{d}}$.
It is the degree-zero part of~$\mathbb{C}[\repspace{Q,\vect{d}}]^{\GL_{\vect{d}}}_\Theta$,
thus there exists a natural projective map
\begin{equation}
  \pi\colon\modulispace[\Theta\semistable]{Q,\vect{d}}\rightarrow\modulispace[0\semistable]{Q,\vect{d}}.
\end{equation}
In particular, if~$Q$ is acyclic,~$\modulispace[\Theta\semistable]{Q,\vect{d}}$
is a projective variety.
In general, we denote by
\begin{equation}
  \modulispace[\Theta\semistable,\nilpotent]{Q,\vect{d}}=\pi^{-1}(0)
\end{equation}
the zero fibre, which is thus a projective variety,
parametrizing isomorphism classes of~$\Theta$-polystable representations~$V$
which are \emph{nilpotent} in the sense that all oriented cycles in~$Q$
are represented by nilpotent compositions of the~$V_\alpha$ in~$V$.

\paragraph{Wall-and-chamber structure}
For a proper non-zero dimension vector~$\vect{e}\leq\vect{d}$, we consider the hyperplane
\begin{equation}
  \mathrm{W}_{\vect{e}}=\{\Theta\in\Stab(\vect{d})\mid \Theta(\vect{e})=0\}.
\end{equation}
We call~$\mathrm{W}_{\vect{e}}$ a proper wall in~$\Stab(\vect{d})$
if, for every~$\Theta\in\mathrm{W}_{\vect{e}}$,
there exist~$\Theta$-semistable representations~$W$ and~$W'$
of dimension vector~$\vect{e}$ and~$\vect{d}-\vect{e}$,
respectively (in this case,~$W\oplus W'$ is a~$\Theta$-semistable non-stable representation of dimension vector~$\vect{d}$).
Inside~$\mathbb{R}\otimes_\mathbb{Q}\Stab(\vect{d})$,
we consider the real hyperplane arrangement given by the~$\mathbb{R}\otimes_\mathbb{Q}\mathrm{W}_{\vect{e}}$
for~$\mathrm{W}_{\vect{e}}$ a proper wall.
The intersection of the chambers of this arrangement
(that is, the connected components of the complement of all hyperplanes)
with~$\Stab(\vect{d})$ are called the chambers of~$\Stab(\vect{d})$.
If two stability parameters belong to the same chamber, the resulting semistable loci in~$\repspace{Q,\vect{d}}$ are identical,
thus their respective moduli spaces are isomorphic.

If~$\Theta\in\Stab(\vect{d})$ does not belong to a proper wall,
then there is an equality~$\repspace[\Theta\semistable]{Q,\vect{d}}=\repspace[\Theta\stable]{Q,\vect{d}}$,
and, consequently,~$\modulispace[\Theta\semistable]{Q,\vect{d}}=\modulispace[\Theta\stable]{Q,\vect{d}}$
is an irreducible smooth variety of dimension~$1-\langle\vect{d},\vect{d}\rangle$.
For more on the wall-and-chamber decomposition for quiver moduli,
see \cite{MR5007902}.

\paragraph{Invariants of quiver moduli}
The explicit combinatorial nature of quiver moduli and their origin in linear algebra
makes them particularly interesting and tractable varieties to compute invariants of.
Beyond the use of local quivers to study singularities and fibres,
we will also freely use the following results on determining invariants for quiver moduli:
\begin{itemize}
  \item Betti numbers \cite{MR1974891};
  \item Chow rings \cite{MR3318266,MR4770368};
  \item the cohomology of the tangent bundle \cite{MR4883104,MR4954467};
  \item the Picard group \cite{MR4352662};
  \item the position of~$\Theta$ in the wall-and-chamber decomposition \cite{MR5007902}.
\end{itemize}
These (and other) invariants can be computed using \textsc{QuiverTools} \cite{quivertools,MR5076554}.

\paragraph{Fano quiver moduli}
Our main interest in this article is quiver moduli which happen to be Fano.
For this, we recall some notions from \cite{MR4352662,MR4954467}.
We say that~$\vect{d}$ is \emph{strongly amply stable}
for the stability parameter~$\Theta$ if
\begin{equation}
  \langle\vect{e},\vect{d}-\vect{e}\rangle\leq -2
\end{equation}
for all proper non-zero~$\vect{e}\leq\vect{d}$
such that~$\Theta(\vect{e})>0$.
We say that~$\vect{d}$ is \emph{amply stable}
for the stability parameter~$\Theta$
if the complement of~$\repspace[\Theta\semistable]{Q,\vect{d}}$ in~$\repspace{Q,\vect{d}}$
has codimension at least two.
Strong ample stability is a sufficient numerical criterion
for ample stability,
which is easier to check in practice,
as it does not require understanding the Harder--Narasimhan stratification.

We also have a preferred choice of stability parameter for a given dimension vector~$\vect{d}$,
namely the one given by
\begin{equation}
  \Theta_\can=\langle\vect{d},\_\rangle-\langle\_,\vect{d}\rangle
\end{equation}
which we call the \emph{canonical stability parameter}.

Under these conditions,
the following theorem summarizes several results in the literature.

\begin{theorem}
  \label{theorem:fano-quiver-moduli}
  Assume the following three conditions:
  \begin{itemize}
    \item $Q$ is acyclic,
    \item $\Theta_\can$ does not belong to a proper wall,
    \item and~$\vect{d}$ is~$\Theta_\can$-amply stable.
  \end{itemize}
  Then~$\modulispace[\Theta_\can\semistable]{Q,\vect{d}}$ is
  a smooth projective Fano variety of dimension~$1-\langle\vect{d},\vect{d}\rangle$,
  Picard rank~$|\supp(\vect{d})|-1$
  and index~$\gcd(\Theta_\can)=\gcd\{\Theta_i\mid i\in Q_0\}$.
  It is rational, of pure Hodge--Tate type,
  its space of global vector fields is isomorphic to~$\mathrm{HH}^1(\mathbb{C}Q)$,
  and it is infinitesimally rigid.
\end{theorem}
The Fano property, Picard rank, and index are~\cite[Theorem~4.2]{MR4352662}.
Rationality is~\cite[Theorem~6.4]{MR1914089}.
The pure Hodge--Tate type is~\cite[Theorem~3]{MR1324213}.
The global vector fields isomorphism is~\cite[Theorem~B]{MR4883104},
whilst Happel's computation of~$\mathrm{HH}^1$ for path algebras is in~\cite{MR1035222}.
Infinitesimal rigidity corresponds to~\cite[Corollary~D]{MR4954467}.

Note that the standing assumption in \cite{MR4883104,MR4954467}
is that of \emph{strong} ample stability
(or as explained in \cite[Remark~1.3]{MR4954467},
it suffices that a certain inequality holds
for every non-trivial Harder--Narasimhan type).
This condition is replaced by the weaker and easier condition
of ample stability in \cite{ample-stability}.

\section{A toolkit for quiver moduli}
\label{section:techniques}
We will use several techniques to study quiver moduli,
whose statements we now recall.

\subsection{Identifications between quiver moduli}
There exist various identifications between quiver moduli,
induced by operations on the quiver, or the dimension vector.

\paragraph{Linear duality}
The first identification is standard.
Let~$Q^{\rm op}$ be the opposite quiver to~$Q$ (that is, reverse all arrows).
\begin{lemma}
  \label{lemma:linear-duality}
  There exists an isomorphism
  \begin{equation}
    \modulispace[\Theta\stableorsemistable]{Q,\vect{d}}
    \cong
    \modulispace[(-\Theta)\stableorsemistable]{Q^{\rm op},\vect{d}},
  \end{equation}
  given by dualizing all maps representing the arrows in a quiver representation.
\end{lemma}

\paragraph{Admissible reflections}
Assume that~$i\in Q_0$ is a sink in~$Q$,
that~$\langle\vect{d},\vect{i}\rangle\leq 0$,
and that~$\Theta_i<0$.
Define the reflected quiver~$\mathrm{s}_i(Q)$ by reversing all arrows pointing to~$i$,
consider the dimension vector~$\mathrm{s}_i(\vect{d})$ for~$\mathrm{s}_i(Q)$,
and define the reflected stability parameter~$\mathrm{s}_i(\Theta)$ by
\begin{equation}
  \mathrm{s}_i(\Theta)=\langle \mathrm{s}_i(\alpha),\_\rangle_{\mathrm{s}_i(Q)}\mbox{ if }\Theta=\langle\alpha,\_\rangle_Q.
\end{equation}
Then we have the following,
see \cite[Theorem~3.2]{MR5005835},
or \cite[\S2.4]{10.4171/JCA/97}.
\begin{proposition}
  There exists an isomorphism
  \begin{equation}
    \modulispace[(\mathrm{s}_i(\Theta))\semistable]{\mathrm{s}_i(Q),\mathrm{s}_i(\vect{d})}
    \cong
    \modulispace[\Theta\semistable]{Q,\vect{d}}.
  \end{equation}
\end{proposition}
Using \cref{lemma:linear-duality},
we have an analogous result for~$i\in Q_0$ a source of~$Q$.

We say that two pairs~$(Q,\vect{d})$ and~$(Q',\vect{d}')$ are \emph{related by admissible reflections}
if one can be obtained from the other by a sequence of admissible reflections and dualities.

\begin{remark}
  There is no known identification of moduli spaces for~$i$ being neither sink nor source.
\end{remark}

\subsection{Local quivers}
Quiver moduli are automatically smooth in stable points,
because the global dimension of the category of representations is~1.
At strictly semistable points in~$\modulispace[\Theta\semistable]{Q,\vect{d}}$
one can determine the types of singularities
up to \'etale equivalence using Luna's slice theorem,
leading to so-called local quivers \cite{MR1972892}.

Let~$V=\bigoplus_{k=1}^s{U_k}^{m_k}$ be a~$\Theta$-polystable representation,
and assume that the~$U_k$ are chosen to be pairwise non-isomorphic.
Then we call
\begin{equation}
  \xi=((\dimvect(U_k))^{m_k})_k
\end{equation}
the \emph{polystable type} of~$V$.
The set~$S_\xi\subset\modulispace[\Theta\semistable]{Q,\vect{d}}$ of representations
of fixed polystable type~$\xi$ is locally closed,
and the quotient map
\begin{equation}
  \repspace[\Theta\semistable]{Q,\vect{d}}\rightarrow\modulispace[\Theta\semistable]{Q,\vect{d}}
\end{equation}
is \'etale-locally trivial over~$S_\xi$.
We call~$S_\xi$ the \emph{Luna stratum} in~$\modulispace[\Theta\semistable]{Q,\vect{d}}$
associated to the decomposition type~$\xi$.

We then define the \emph{local quiver}~$Q_\xi$ for~$\xi$
as the quiver with vertices~$i_1,\ldots,i_s$,
and with the number of arrows from~$i_k$ to~$i_l$ given by
\begin{equation}
  \delta_{k,l}-\langle\dimvect(U_k),\dimvect(U_l)\rangle_Q.
\end{equation}
We also define the local dimension vector~$\vect{d}_\xi=\sum_{k=1}^s m_k\vect{i}_k$ for~$Q_\xi$.
The following result holds \cite[Theorem~1.1]{MR1972892}.
\begin{proposition}
  \label{proposition:local-quiver}
  The singularity of~$\modulispace[\Theta\semistable]{Q,\vect{d}}$
  at any point in the Luna stratum~$S_\xi$
  is \'etale equivalent to the singularity of~$0$
  (i.e., the representation with all morphisms zero)
  in the moduli space~$\modulispace[0\semistable]{Q_\xi,\vect{d}_\xi}$.
\end{proposition}

\subsection{Framing}
\label{subsection:framing}
Given~$Q$,~$\vect{d}$,
an integral stability parameter~$\Theta\in\Hom(\mathbb{Z}Q_0,\mathbb{Z})$ as before,
and a framing datum~$\vect{n}\in \mathbb{N}Q_0$,
we introduce a new quiver~$\widehat{Q}$
by adding a vertex~$i_0$ to~$Q_0$
and arrows~$\alpha_{i,1},\ldots,\alpha_{i,n_i}\colon i_0\rightarrow i$ for all~$i\in Q_0$ to~$Q$.
We define~$\widehat{\vect{d}}=\vect{d}+\vect{i}_0\in\mathbb{N}\widehat{Q}_0$.
For a choice of a positive~$\kappa\in\Hom(\mathbb{Z}Q_0,\mathbb{Z})$
and a positive integer~$C$ such that~$\kappa(\vect{d})<C\cdot\gcd(\Theta)$,
we define
\begin{equation}
  \widehat{\Theta}\colonequals C\Theta-\kappa+\kappa(\vect{d})\vect{i}_0^*.
\end{equation}
We define
\begin{equation}
  \modulispace[\Theta\framed]{Q,\vect{d},\vect{n}}
  \colonequals
  \modulispace[\widehat{\Theta}\semistable]{\widehat{Q},\widehat{\vect{d}}}
\end{equation}
(which does not depend on the choices of~$\kappa$ and~$C$
since all resulting stability parameters belong to the same chamber in~$\Stab({\widehat{\vect{d}}})$).
It is always smooth and irreducible, and admits a projective map
\begin{equation}
  \label{equation:projection-to-wall}
  p\colon\modulispace[\Theta\framed]{Q,\vect{d},\vect{n}}\rightarrow\modulispace[\Theta\semistable]{Q,\vect{d}}.
\end{equation}
The fibres can be described as follows.
For a polystable type~$\xi=(\vect{d}_k^{m_k})_{k=1}^s$ as before,
we consider the local quiver~$Q_\xi$
and the local dimension vector~$\vect{d}_\xi$ for~$Q_\xi$.
We consider the local framing datum~$\vect{n}_\xi=\sum_{k=1}^s(\vect{n}\cdot\vect{d}_k)\vect{i}_k$.
The following result holds,
see \cite[Theorem~4.1]{MR2511752}.
\begin{proposition}
  \label{proposition:fibres-framed}
  For any~$V\in S_\xi\subset\modulispace[\Theta\semistable]{Q,\vect{d}}$,
  we have an isomorphism
  \begin{equation}
    p^{-1}(V)\cong\modulispace[0\semistable,\nilpotent]{Q_\xi,\vect{d}_\xi,\vect{n}_\xi}.
  \end{equation}
\end{proposition}
A representation of the latter moduli space consists of
a nilpotent representation~$W$ of~$Q_\xi$ of dimension vector~$\vect{d}_\xi$,
together with linear maps~$f_k\colon\mathbb{C}^{\vect{n}\cdot\vect{d}_k}\rightarrow W_{i_k}$ for all~$k$.
Such a tuple is stable if and only if the images of the~$f_k$ generate~$W$ as a representation of~$Q_\xi$.

\subsection{Projection to the wall}
\label{subsection:projection-to-wall}
If a stability parameter~$\overline{\Theta}$ belongs to
the closure of a chamber in~$\Stab(\vect{d})$ containing a stability parameter~$\Theta$ in its interior,
we construct a projection map between the corresponding moduli spaces.

The following lemma is closely related to the methods of \cite{MR5007902}.

\begin{lemma}
  \label{lemma:projection}
  Let~$\Theta,\overline{\Theta}\in\Stab(\vect{d})$ be two stability parameters,
  where~$\Theta$ does not belong to a proper wall.
  Assume that the following property holds:
  if~$\vect{e}\leq\vect{d}$ is a proper non-zero dimension vector such that~$\Theta(\vect{e})<0$,
  but~$\overline{\Theta}(\vect{e})>0$,
  then there exists a dimension vector~$\vect{e}'\hookrightarrow\vect{e}$ such that~$\Theta(\vect{e}')>0$,
  or a dimension vector~$\vect{d}-\vect{e}\twoheadrightarrow\vect{f}'$ such that~$\Theta(\vect{f}')<0$.
  Then~$\overline{\Theta}$ belongs to the closure of the chamber containing~$\Theta$,
  and~$\repspace[\Theta\semistable]{Q,\vect{d}}$ is contained in~$\repspace[\smash{\overline{\Theta}}\semistable]{Q,\vect{d}}$.
\end{lemma}

\begin{proof}
  Assume that~$V$ is a~$\Theta$-(semi-)stable representation of dimension vector~$\vect{d}$
  which is not~$\overline{\Theta}$-semistable.
  Then there exists a proper non-zero subrepresentation~$U\subset V$,
  say of dimension vector~$\vect{e}$,
  such that~$\overline{\Theta}(\vect{e})>0$,
  although~$\Theta(\vect{e})<0$.
  First assume that there exists~$\vect{e}'\hookrightarrow\vect{e}$ such that~$\Theta(\vect{e}')>0$.
  Then~$U$ admits a subrepresentation~$U'$ of dimension vector~$\vect{e}'$,
  which is also a subrepresentation of~$V$,
  and thus~$\Theta(\vect{e}')<0$, a contradiction.
  Now assume that there exists~$\vect{d}-\vect{e}\twoheadrightarrow\vect{f}'$
  such that~$\Theta(\vect{f}')<0$.
  Then~$V/U$ admits a factor representation~$W'$ of dimension vector~$\vect{f}'$,
  which is also a factor representation of~$V$,
  and thus~$\Theta(\vect{f}')>0$, again a contradiction.

  A wall~$\mathrm{W}_{\vect{e}}$ lies on the line segment between~$\Theta$ and~$\overline{\Theta}$
  if and only if
  (without loss of generality, otherwise replace~$\vect{e}$ by~$\vect{d}-\vect{e}$)
  we have~$\Theta(\vect{e})<0$, but~$\overline{\Theta}(\vect{e})>0$.
  Now the same argument as above excludes such a wall.
\end{proof}

It follows from the final implication that, in this situation, we have a well-defined projective map
\begin{equation}
  p\colon\modulispace[\Theta\semistable]{Q,\vect{d}}\rightarrow\modulispace[\overline{\Theta}\semistable]{Q,\vect{d}}.
\end{equation}
In fact, the map~$p$ from \cref{subsection:framing}
as well as the map~$\pi$ to the moduli space of semisimple representations,
can be interpreted as special cases of this construction.

Again, the fibres can be described using local quivers.
For a~$\overline{\Theta}$-polystable type~$\xi=(\vect{d}_k^{m_k})_{k=1}^s$,
we consider the local quiver~$Q_\xi$
and the local dimension vector~$\vect{d}_\xi$ for~$Q_\xi$.
We consider the stability parameter
\begin{equation}
  \Theta_\xi\colonequals\sum_{k=1}^s\Theta(\vect{d}_k)\vect{i}_k^*.
\end{equation}
The following result holds,
see \cite[Theorem 3.5]{MR3644807}.
\begin{proposition}
  \label{proposition:fibres-projection-wall}
  For any~$V\in S_\xi\subset\modulispace[\overline{\Theta}\semistable]{Q,\vect{d}}$,
  we have an isomorphism
  \begin{equation}
    p^{-1}(V)\cong\modulispace[\Theta_\xi\semistable,\nilpotent]{Q_\xi,\vect{d}_\xi}.
  \end{equation}
\end{proposition}

\section{Classification of Fano fourfolds from subspace quivers}
\label{section:classification}
In this section we will classify Fano fourfolds arising from subspace quivers,
under the assumption that the dimension vector lives in
(or can be reflected into) the fundamental domain~$\mathrm{F}_Q$.
Without this assumption, the classification problem becomes very impractical to perform.

\subsection{Moduli spaces for subspace quivers}
Let~$Q^{(m)}$ be the~$m$-subspace quiver with set of vertices
\begin{equation}
  Q^{(m)}_0=\{i_1,\ldots,i_m,j\}
\end{equation}
and set of arrows
\begin{equation}
  Q^{(m)}_1=\{(\alpha_1\colon i_1\rightarrow j),\ldots,(\alpha_m\colon i_m\rightarrow j)\}.
\end{equation}
We consider dimension vectors
\begin{equation}
  \vect{d}=\sum_{k=1}^me_k\vect{i}_k+d\vect{j}
\end{equation}
for which we can assume, without loss of generality, by permuting the vertices~$i_1,\ldots,i_m$, that
\begin{equation}
  0<e_1\leq\ldots\leq e_m<d.
\end{equation}
We abbreviate the dimension vector~$\vect{d}=(e_1,\dots,e_m;d)$
as~$\vect{d}=(1^{k_1},2^{k_2},\ldots;d)$,
where in the tuple of~$e_k$, an entry~$e$ appears with multiplicity~$k_e$.

We consider, again without loss of generality, stability parameters
\begin{equation}
  \Theta=\sum_{k=1}^m\Theta_{i_k}\vect{i}_k^*+\Theta_j\vect{j}^*\in\Stab(\vect{d})
\end{equation}
such that~$\Theta_{i_k}>0$ for all~$k=1,\ldots,m$,
and we moreover assume (after translation) that
\begin{equation}
  \sum_{k=1}^m\Theta_{i_k}e_k+\Theta_jd=0.
\end{equation}

\begin{notation}
  We denote the resulting moduli spaces~$\modulispace[\Theta_\can\semistable]{Q^{(m)},\vect{d}}$
  by~$\modulispace{e_1,\ldots,e_m;d}$,
  or the shorthand~$\modulispace{1^{k_1},2^{k_2},\ldots;d}$.
\end{notation}

Before we classify and study these moduli spaces,
let us give an alternative description,
used already in \cref{theorem:main}.
For this we need the following easy lemma.
\begin{lemma}
  \label{lemma:stable-implies-injective}
  Under the above assumptions on~$\vect{d}$ and~$\Theta$,
  let~$V$ be a~$\Theta$-stable representation of dimension vector~$\vect{d}$.
  Then for all~$\alpha_k\colon i_k\to j$, with~$k=1,\ldots,m$,
  the morphism~$V(\alpha_k)\colon V_{i_k}\to V_j$ is injective.
\end{lemma}

\begin{proof}
  If~$V(\alpha_k)$ is not injective,
  then the representation~$K$ which has~$\ker V(\alpha_k)$ at vertex~$i_k$,
  and~0 for all other vertices is a subrepresentation of~$V$.
  But~$\Theta(\dimvect(K))>0$ by our assumption on~$\Theta$,
  contradicting stability of~$V$.
\end{proof}

This allows us to give the following geometric interpretation.

\begin{corollary}
  \label{corollary:git-description}
  The moduli space~$\modulispace{e_1,\ldots,e_m;d}$
  is isomorphic to a GIT quotient
  for the diagonal action of~$\PGL(V)$ on the product of Grassmannians
  \begin{equation}
    \prod_{k=1}^m\Gr(e_k,V),
  \end{equation}
  where~$V$ is a~$d$-dimensional~$\mathbb{C}$-vector space,
  a tuple of subspaces~$(U_1,\ldots,U_m)$
  is taken to be~$\Theta$-semistable if
  \begin{equation}
    \label{equation:stability}
    \frac{\dim\sum_{k=1}^mU_k'}{d}\geq\frac{\sum_{k=1}^m\Theta_{i_k}\dim U_k'}{\sum_{k=1}^m\Theta_{i_k}e_k}
  \end{equation}
  for all tuples of subspaces~$(U_k'\subset U_k)_{k=1}^m$,
  and it is~$\Theta$-stable if the strict inequality holds for all proper non-zero such tuples, and~$\sum_{k=1}^mU_k=V$.
\end{corollary}

There is the following interpretation of the non-emptiness of such a GIT quotient.
\begin{lemma}
  \label{lemma:stable-general-position}
  Assuming such a tuple~$(U_1,\ldots,U_m)$ of subspaces of~$V$ to be in general position,
  in the sense that~$U_1'+\cdots+U_m'$ is a direct sum in~$V$ whenever~$\sum_{k=1}^m\dim U_k'\leq d$,
  a stable point exists if
  \begin{equation}
    0<\Theta_{i_k}<\frac{\sum_{l=1}^m\Theta_{i_l}e_l}{d}
  \end{equation}
  for all~$k=1,\ldots,m$.
\end{lemma}

\begin{proof}
  First assume that~$\sum_{k=1}^m\dim U_k'\leq d$ for a non-zero tuple~$(U_1',\ldots,U_m')$, thus
  \begin{equation}
    \dim\sum_{k=1}^mU_k'=\sum_{k=1}^m\dim U_k'.
  \end{equation}
  Then the assumptions on~$\Theta$ imply
  \begin{equation}
    0<\sum_{k=1}^m\left( \sum_{l=1}^m\Theta_{i_l}e_l-d\Theta_{i_k} \right)\dim U_k'
  \end{equation}
  which can be rewritten as
  \begin{equation}
    d\sum_{k=1}^m\Theta_{i_k}\dim U_k'<\sum_{k=1}^m\dim U_k'\cdot\sum_{l=1}^m\Theta_{i_l}e_l,
  \end{equation}
  and thus
  \begin{equation}
    \frac{\sum_{k=1}^m\Theta_{i_k}\dim U_k'}{\sum_{k=1}^m\Theta_{i_k}e_k}<\frac{\sum_{k=1}^m\dim U_k'}{d}=\frac{\dim\sum_{k=1}^mU_k'}{d}.
  \end{equation}
  Now assume~$\sum_{k=1}^m\dim U_k'\geq d$ for a proper tuple~$(U_1',\ldots,U_m')$, which implies~$\sum_{k=1}^mU_k'=V$. Then
  \begin{equation}
    0<\sum_{k=1}^m\Theta_{i_k}(e_k-\dim U_k'),
  \end{equation}
  as the tuple of subspaces is assumed to be proper, so not all~$U_k'=U_k$,
  whilst~$\Theta_{i_k}>0$ and~$\dim U_k'\leq e_k$ for all~$k$.
  Rearranging this to~$\sum_{k=1}^m\Theta_{i_k}\dim U_k'<\sum_{k=1}^m\Theta_{i_k}e_k$
  and dividing by the positive number~$\sum_{k=1}^m\Theta_{i_k}e_k$,
  we thus obtain the strict inequality
  \begin{equation}
    \frac{\sum_{k=1}^m\Theta_{i_k}\dim U_k'}{\sum_{k=1}^m\Theta_{i_k}e_k}<1=\frac{\dim\sum_{k=1}^mU_k'}{d},
  \end{equation}
  as required for stability in \cref{corollary:git-description}.
\end{proof}

For later use we also state and prove
the following general method for identifying framed moduli spaces
amongst the~$\modulispace{e_1,\ldots,e_m;d}$.

\begin{proposition}
  \label{proposition:subspace-framed}
  If~$d$ divides~$\sum_{k=1}^me_k$,
  then~$\modulispace[\Theta_\can\framed]{Q^{(m)},\vect{d},\vect{j}}$ is isomorphic to~$\modulispace{1,e_1,\ldots,e_m;d}$,
  which uses the quiver~$Q^{(m+1)}$.
\end{proposition}

\begin{proof}
  For the quiver~$Q^{(m)}$ and~$\vect{d}=(e_1,\ldots,e_m;d)$,
  we have that~$\Theta_\can=d\sum_{k=1}^m\vect{i}_k^*-\sum_{k=1}^me_k\cdot\vect{j}^*$.
  We consider~$\kappa=d\sum_{k=1}^m\vect{i}_k^*+\vect{j}^*$ and~$C=\sum_{k=1}^me_k+2$.
  Then the conditions on~$\Theta$,~$\kappa$ and~$C$ of the construction of framed moduli spaces are fulfilled since
  \begin{equation}
    \kappa(\vect{d})=d\sum_{k=1}^me_k+d<\left( \sum_{k=1}^me_k+2 \right)d=C\gcd(\Theta_\can)
  \end{equation}
  by assumption.
  By definition, we then have~$\widehat{Q}=Q^{(m+1)}$,~$\widehat{\vect{d}}=(1,e_1,\ldots,e_m;d)$
  and~$\widehat{\Theta}=(\sum_{k=1}^me_k+1)\cdot\Theta'$,
  where~$\Theta'$ is the canonical stability parameter for~$\widehat{\vect{d}}$.
\end{proof}

Finally,
we observe the following,
addressing one of the special features of subspace quiver moduli.
\begin{lemma}
  \label{lemma:no-vector-fields}
  Let~$\vect{d}=(e_1,\ldots,e_m;d)$ be an indivisible dimension vector for
  the subspace quiver~$Q^{(m)}$,
  and let~$\Theta$ be a stability parameter.
  Suppose moreover that
  \begin{itemize}
    \item $\Theta$ does not belong to a proper wall, and
    \item $\vect{d}$ is~$\Theta$-amply stable.
  \end{itemize}
  Then the moduli space~$X=\modulispace[\Theta\semistable]{Q^{(m)},\vect{d}}$
  admits no non-trivial global vector fields,
  i.e.,~$\mathrm{H}^0(X,\mathrm{T}_X)=0$.
\end{lemma}

\begin{proof}
  The conditions on~$\vect{d}$ and~$\Theta$
  correspond to \cite[Assumption~2.4]{MR4883104},
  so that we can apply \cite[Theorem~B]{MR4883104} and get
  \begin{equation}
    \mathrm{H}^0(X,\mathrm{T}_X)\cong\mathrm{Lie}\,\mathrm{Aut}^0(X)\cong\mathrm{HH}^1(\mathbb{C}Q^{(m)}).
  \end{equation}
  Happel~\cite{MR1035222} shows that for the path algebra~$\mathbb{C}Q$
  of a connected acyclic quiver~$Q$,
  \begin{equation}
    \dim\mathrm{HH}^1(\mathbb{C}Q)
    = 1 - \#Q_0 + \sum_{\alpha\in Q_1} \dim e_{\operatorname{t}(\alpha)}\mathbb{C}Qe_{\operatorname{s}(\alpha)},
  \end{equation}
  where~$\dim e_{\operatorname{t}(\alpha)}\mathbb{C}Qe_{\operatorname{s}(\alpha)}$
  counts directed paths from~$\operatorname{s}(\alpha)$ to~$\operatorname{t}(\alpha)$.
  In particular,~$\mathrm{HH}^1(\mathbb{C}Q)=0$ whenever the underlying graph is a tree,
  which applies to the subspace quiver~$Q^{(m)}$.
\end{proof}

\subsection{Classification of fourfolds}
Ideally, our aim would be to classify all dimension vectors~$\vect{d}$ as above
for which we can identify the moduli space~$\modulispace{e_1,\ldots,e_m;d}$
as a smooth projective Fano variety of fixed dimension (i.e., dimension~4~in the present paper).
It suffices to classify such tuples up to admissible reflection.
However, already for the case~$m=5$,
computer experiments provide an uncontrollable wealth of such~$\vect{d}$.
On the other hand,
there also seem to be many isomorphisms between the resulting moduli spaces
that cannot be detected by the general theory described in \cref{section:techniques},
drastically collapsing the number of resulting Fano 4-fold quiver moduli.
As a consequence, we have to make a reasonable additional assumption.

Our choice for this assumption is motivated by the discrepancy between admissible reflections and all reflections:
we require that~$\vect{d}$ is related by admissible reflections to a dimension vector in the fundamental domain~$\mathrm{F}_{Q^{(m)}}$.

Computing~$(\vect{d},\vect{i})$ for each vertex~$i$ of~$Q^{(m)}$,
one finds that~$(\vect{d},\vect{i}_k)=2e_k-d$ for~$k=1,\ldots,m$
whilst we have~$(\vect{d},\vect{j})=2d-\sum_{k=1}^me_k$ for the central vertex.
Since all~$e_k\geq 1$, the support of~$\vect{d}$ is automatically connected.
Thus the condition that~$\vect{d}$ belongs to the fundamental domain~$\mathrm{F}_{Q^{(m)}}$ reads
\begin{equation}
  2e_k\leq d\mbox{ for }k=1,\ldots,m\mbox{ and } 2d\leq\sum_{k=1}^me_k.
\end{equation}

\begin{proposition}
  If a dimension vector~$\vect{d}$ for~$Q^{(m)}$ as above
  belongs to~$\mathrm{F}_{Q^{(m)}}$
  and~$\langle\vect{d},\vect{d}\rangle=-3$,
  then~$\vect{d}$ is one of the following:
  \begin{itemize}
    \item~$\vect{d}=(1,2^4;4)$ for~$Q^{(5)}$,
    \item~$\vect{d}=(1^6;3)$ for~$Q^{(6)}$,
    \item~$\vect{d}=(1^7;2)$ for~$Q^{(7)}$.
  \end{itemize}
\end{proposition}

\begin{proof}
  The conditions~$e_k\geq 1$ and~$2e_k\leq d$ exclude the case~$d=1$.

  If~$d=2$, then~$e_k=1$ for all~$k=1,\ldots,m$,
  and the condition~$\langle\vect{d},\vect{d}\rangle=-3$ immediately yields~$m=7$,
  and we arrive at the third case.

  If~$d=3$, we still have~$e_k=1$ for all~$k$,
  and again~$\langle\vect{d},\vect{d}\rangle=-3$ yields~$m=6$, thus the second case.

  If~$d=4$, we have~$e_k=1$ or~$e_k=2$,
  and we assume that~$e_1=\dots=e_{m'}=1$ and~$e_{m'+1}=\ldots=e_m=2$.
  Then~$\langle\vect{d},\vect{d}\rangle=-3$ translates to~$m'=4m-19$,
  whereas~$2d\leq\sum_{k=1}^me_k$ translates to~$m'\leq 2m-8$.
  We thus find~$m=5$ and~$m'=1$, thus arriving at the first case.

  It thus suffices to exclude the possibility~$d\geq 5$.
  We can rewrite the condition~$\langle\vect{d},\vect{d}\rangle=-3$ as
  \begin{equation}
    \sum_{k=1}^me_k(d-2e_k)+d\left( \sum_{k=1}^me_k-2d \right)=6,
  \end{equation}
  noting that all summands are nonnegative and at most~$6$.

  First assume that~$2d<\sum_{k=1}^me_k$, in which case~$d(\sum_{k=1}^me_k-2d)\leq 6$,
  thus~$d=5$ or~$d=6$ by assumption.

  If~$d=5$, we conclude~$\sum_{k=1}^me_k=11$.
  On the other hand, we have
  \begin{equation}
    1=\sum_{k=1}^m e_k(d-2e_k)\geq\sum_{k=1}^me_k,
  \end{equation}
  a contradiction.

  Similarly,~$d=6$ yields~$\sum_{k=1}^me_k=13$.
  On the other hand, we have
  \begin{equation}
    0=\sum_{k=1}^me_k(d-2e_k),
  \end{equation}
  thus~$e_k=3$ for all~$k$, again yielding a contradiction.

  Thus we can assume~$2d=\sum_{k=1}^me_k$, thus~$\sum_{k=1}^me_k(d-2e_k)=6$.

  If~$d$ is odd, we have~$d-2e_k\geq 1$ for all~$k$, thus
  \begin{equation}
    6=\sum_{k=1}^me_k(d-2e_k)\geq\sum_{k=1}^me_k=2d,
  \end{equation}
  and thus~$d\leq 3$, a contradiction.

  So assume that~$d$ is even.
  There exists an index~$1\leq k\leq m$ such that~$d-2e_k>0$ and~$e_k(d-2e_k)\leq 6$.
  We find the following possibilities for the triple~$(e_k,d,e_k(d-2e_k))$:
  \begin{equation}
    (1,6,4),\;(2,6,4),\;(1,8,6),\;(3,8,6).
  \end{equation}
  Consequently,~$d=6$ or~$d=8$.

  If~$d=6$, then~$e_k\leq 3$ for all~$k$, and all summands~$e_k(d-2e_k)$ are either~$4$ or~$0$, a contradiction to such summands summing up to~$6$.

  If~$d=8$, then~$e_k\leq 4$ and~$e_k\not=2$ by the above list, and we find the following pairs~$(e_k,e_k(d-2e_k)):$
  \begin{equation}
    (1,6),(3,6),(4,0).
  \end{equation}
  The conditions~$\sum_{k=1}^me_k=16$ and~$\sum_{k=1}^me_k(d-2e_k)=6$ again yield a contradiction.
\end{proof}

\begin{theorem}
  \label{theorem:classification}
  Assume that a dimension vector~$\vect{d}$ for~$Q^{(m)}$ as above is given such that:
  \begin{itemize}
    \item~$\vect{d}$ is related by admissible reflections to a dimension vector belonging to~$\mathrm{F}_{Q^{(m)}}$,
    \item~$\langle\vect{d},\vect{d}\rangle=-3$,
    \item~$\Theta_\can\in\Stab(\vect{d})$ does not belong to a proper wall,
    \item~$\vect{d}$ is (strongly)~$\Theta_\can$-amply stable.
  \end{itemize}
  Then, up to admissible reflections,~$\vect{d}$ is one of the following:
  \begin{itemize}
    \item~$\vect{d}=(1,2^4;4)$ for~$Q^{(5)}$,
    \item~$\vect{d}=(1^5,2;3)$ for~$Q^{(6)}$,
    \item~$\vect{d}=(1^4,2^2;3)$ for~$Q^{(6)}$,
    \item~$\vect{d}=(1^7;2)$ for~$Q^{(7)}$.
  \end{itemize}
\end{theorem}

\begin{proof}
  All that has to be done is to take one of the dimension vectors of the previous proposition,
  choose an arbitrary orientation of the graph~$|Q^{(m)}|$ underlying~$Q^{(m)}$,
  test whether the above conditions are fulfilled,
  and provide a sequence of admissible reflections to one of the four cases in the list.

  In this strategy,
  we use that admissible reflections and linear duality
  preserve the conditions of the theorem,
  so that they can be tested on any representative.
  Indeed,
  a direct computation shows that a reflection at a sink~$i$
  leaves the Euler form invariant,
  \begin{equation}
    \langle\mathrm{s}_i(\vect{e}),\mathrm{s}_i(\vect{e}')\rangle_{\mathrm{s}_i(Q)}
    =\langle\vect{e},\vect{e}'\rangle_Q,
  \end{equation}
  whilst linear duality interchanges it with its opposite.
  Hence~$\langle\vect{d},\vect{d}\rangle$ is preserved,
  and the canonical stability parameter is mapped to
  the canonical stability parameter,
  as~$\mathrm{s}_i(\Theta)(\mathrm{s}_i(\vect{e}))=\Theta(\vect{e})$.
  Finally,
  the isomorphism of \cite[Theorem~3.2]{MR5005835}
  (see also \cite[Theorem~2.5]{10.4171/JCA/97})
  matches polystable representations,
  with dimension vectors of the summands transforming by~$\mathrm{s}_i$,
  so that the condition that~$\Theta_\can$ does not lie on a proper wall
  is preserved as well.

  Let us start with~$\vect{d}=(1,2^4;4)$.
  Since reflection at an entry~$2$ does not change~$\vect{d}$,
  we can assume four arrows to be oriented from the entry~$2$ to the entry~$4$.
  If the last arrow is oriented from~$4$ to~$1$,
  we can apply four admissible reflections at the entries~$2$,
  followed by linear duality,
  to arrive at the case where all arrows are oriented towards the entry~$4$.
  The canonical stability parameter is then~$\Theta_\can=4\sum_{k=1}^5\vect{i}_k^*-9\vect{j}^*$,
  and the two conditions on this stability parameter can be verified directly.
  Here, and in the next cases, one can by hand verify that the setting
  is strongly amply stable, and thus amply stable.
  There will never be a case excluded because it is
  not strongly amply stable.

  Let us continue with~$\vect{d}=(1^7;2)$.
  Since reflection at an entry~$1$ does not change~$\vect{d}$,
  we can assume all seven arrows to be oriented towards the entry~$2$.
  The canonical stability parameter is then~$\Theta_\can=2\sum_{k=1}^7\vect{i}_k^*-7\vect{j}^*$,
  and again the two conditions on~$\Theta_\can$ can be verified directly.

  Finally, we consider~$\vect{d}=(1^6;3)$.
  Up to linear duality,
  we have to consider the four cases where~$3$,~$4$,~$5$ or~$6$ arrows are oriented towards the entry~$3$.
  By reflections at the entries~$1$ at sinks,
  we can reflect~$\vect{d}$ to the dimension vectors
  \begin{equation}
    (1^3,2^3;3), \; (1^4,2^2;3),\; (1^5,2;3),\; (1^6;3)
  \end{equation}
  for~$Q^{(6)}$, respectively.
  We compute
  \begin{equation}
    \Theta_\can=3\sum_{k=1}^6\vect{i}_k^*-a\vect{j}^*
  \end{equation}
  for~$a=9,8,7,6$, respectively.
  In the first case,~$\Theta_\can$ belongs to the proper wall~$\mathrm{W}_{\vect{e}}$ for~$\vect{e}=(1,1,1,0,0,0;1)$,
  and in the fourth case~$\Theta_\can$ belongs to the proper wall~$\mathrm{W}_{\vect{e}}$ for~$\vect{e}=(1,1,0,0,0,0;1)$.
  In the second and third case, we can again verify directly the two desired conditions on~$\Theta_\can$.
\end{proof}

By \cref{theorem:fano-quiver-moduli},
the moduli spaces associated to the four dimension vectors in \cref{theorem:classification}
and their canonical stability parameters
are smooth projective Fano fourfolds.
They are labelled~$\QM{1},\ldots,\QM{4}$ in \cref{theorem:main}.
In particular, by~\cref{lemma:no-vector-fields}, all four Fano fourfolds admit no non-trivial global vector fields.

\section{The variety \QM{1}: involution surface bundle}
\label{section:QM1}

We now consider in detail the geometry of the four varieties in the previous theorem,
by constructing explicit maps to known spaces and analyzing their fibres.
We start with~$\QM{1}$, which has Betti numbers~$(1,5,13,5,1)$.

\subsection{Framed structure}
By~\cref{proposition:subspace-framed},
the moduli space~$\QM{1}$ is isomorphic to the framed moduli space of~$\modulispace{2,2,2,2;4}$
(i.e., a quiver moduli space for the subspace quiver~$Q^{(4)}$)
with framing datum~$\vect{j}$.
For the dimension vector~$(2,2,2,2;4)$,
there are no stable representations;
a general representation of this dimension vector
is the direct sum of two stable representations of dimension vector~$\vect{d}\colonequals(1,1,1,1;2)$,
where~$Q^{(4)}$ is also the extended Dynkin quiver~$\widetilde{\mathrm{D}}_4$,
which is thus of tame type.
The moduli space~$\modulispace{\vect{d}}$ is isomorphic to~$\mathbb{P}^1$,
the isomorphism being given by the cross-ratio of four points in~$\mathbb{P}^1$.
We see that
\begin{equation}
  \modulispace{2\vect{d}}\cong\Sym^2\modulispace{\vect{d}}\cong\Sym^2\mathbb{P}^1\cong\mathbb{P}^2.
\end{equation}
Thus, by \cref{proposition:subspace-framed} and the analysis of the framing construction
we obtain the following.
\begin{lemma}
  \label{lemma:qm-1-projection}
  There exists a surjective morphism
  \begin{equation}
    \label{equation:involution-surface-bundle}
    p\colon\QM{1}\rightarrow\mathbb{P}^2.
  \end{equation}
\end{lemma}
In \cref{subsection:fibre-types} we determine all fibres of~$p$,
in \cref{lemma:first-fiber,lemma:second-type,lemma:third-type,lemma:fourth-type,lemma:fifth-type}.
As all fibres turn out to be two-dimensional,
and both~$\QM{1}$ and~$\mathbb{P}^2$ are smooth,
miracle flatness implies that~$p$ is a flat morphism.

\subsection{Fibre types}
\label{subsection:fibre-types}
To analyze the fibres in \eqref{equation:involution-surface-bundle},
we determine the Luna strata for~$\modulispace{2\vect{d}}$.
The only smaller dimension vectors~$\vect{e}$ with~$\Theta_\can(\vect{e})=0$ admitting stable representations are
\begin{itemize}
  \item~$\vect{d}$ (for which there exists a~$(\mathbb{P}^1\setminus\{0,1,\infty\})$-family of stables), and
  \item~$\vect{e}_K=\vect{i}_k+\vect{i}_l+\vect{j}$ for~$K=\{k<l\}\subset\{1,2,3,4\}$,
    for which there exists a unique representation up to isomorphism.
\end{itemize}
In each case the existence of stables can also be verified using~\cref{lemma:stable-general-position}.
The possible decomposition types are thus
\begin{itemize}
  \item $\xi_1=(\vect{d}^1,\vect{d}^1)$,
  \item $\xi_2=(\vect{d}^1,\vect{e}_K^1,\vect{e}_{\overline{K}}^1)$ for~$K$ as before,
  \item $\xi_3=(\vect{e}_K^1,\vect{e}_{\overline{K}}^1,\vect{e}_L^1,\vect{e}_{\overline{L}}^1)$ for~$K$ and~$L$ as before such that~$|K\cap L|=1$, and~$\overline{K},\overline{L}$ their respective complements,
  \item $\xi_4=(\vect{d}^2)$, and
  \item $\xi_5=(\vect{e}_K^2,\vect{e}_{\overline{K}}^2)$ for~$K$ as before.
\end{itemize}
For these five types, we now describe the fibres using~\cref{proposition:fibres-framed}
as framed moduli spaces of nilpotent representations for local quivers.

\begin{lemma}
  \label{lemma:first-fiber}
  For~$V\in S_{\xi_1}\subset\modulispace{2\vect{d}}\cong\mathbb{P}^2$,
  the fibre~$p^{-1}(V)$ is isomorphic to a smooth quadric surface.
\end{lemma}

\begin{proof}
  For the decomposition type~$\xi_1=(\vect{d}^1,\vect{d}^1)$
  we find the local quiver~$Q_{\xi_1}$
  with dimension vector~$\vect{i}_1+\vect{i}_2$ and framing datum~$2(\vect{i}_1+\vect{i}_2)$.
  We thus have to determine the moduli space of nilpotent semistable representations of the framing of the local quiver,
  for dimension vector~$\vect{i}_0+\vect{i}_1+\vect{i}_2$ and stability~$2\vect{i}_0^*-\vect{i}_1^*-\vect{i}_2^*$.
  Such a representation is of the form
  \begin{equation}
    \begin{tikzpicture}[node distance = 2.5cm]
      \node (a) {$\mathbb{C}$};
      \node (b) [right of = a] {$\mathbb{C}$};
      \node (c) at ($(a)!0.5!(b) + (0,1.1cm)$) {$\mathbb{C}$};
      \draw[->, loop left,  looseness = 20] (a) edge node [left]  {0} (a);
      \draw[->, loop right, looseness = 20] (b) edge node [right] {0} (b);
      \draw[->, bend right = 20] (c) edge node [left]  {$a$} (a);
      \draw[->, bend left  = 20] (c) edge node [right] {$b$} (a);
      \draw[->, bend left  = 20] (c) edge node [right] {$c$} (b);
      \draw[->, bend right = 20] (c) edge node [left]  {$d$} (b);
    \end{tikzpicture}
  \end{equation}
  with~$(a,b)\not=(0,0)\not=(c,d)$.
  We find the semi-invariants~$ac$,~$ad$,~$bc$,~$bd$ of weight~$\Theta$,
  thus the moduli space is the standard quadric in~$\mathbb{P}^3$,
  isomorphic to~$\mathbb{P}^1\times\mathbb{P}^1$.
\end{proof}

\begin{lemma}
  \label{lemma:second-type}
  For~$V\in S_{\xi_2}\subset\modulispace{2\vect{d}}\cong\mathbb{P}^2$,
  the fibre~$p^{-1}(V)$ is isomorphic to the product of a smooth conic with a reduced singular conic.
\end{lemma}

\begin{proof}
  For the decomposition type~$\xi_2=(\vect{d}^1,\vect{e}_K^1,\vect{e}_{\overline{K}}^1)$
  we find the local quiver~$Q_{\xi_2}$
  with dimension vector~$\vect{i}_1+\vect{i}_2+\vect{i}_3$ and the framing datum~$2\vect{i}_1+\vect{i}_2+\vect{i}_3$.
  We thus have to determine the moduli space of nilpotent semistable representations of the framing of the local quiver,
  for dimension vector~$\vect{i}_0+\vect{i}_1+\vect{i}_2+\vect{i}_3$ and stability~$3\vect{i}_0^*-\vect{i}_1^*-\vect{i}_2^*-\vect{i}_3^*$.
  Such a representation is of the form
  \begin{equation}
    \begin{tikzpicture}[node distance = 1.5cm]
      \node (1)                {};
      \node (2) [right of = 1] {};
      \node (3) [right of = 2] {};
      \node (4) [above of = 2] {};
      \draw (1) circle (2pt) node [below] {$\mathbb{C}$};
      \draw (2) circle (2pt) node [below] {$\mathbb{C}$};
      \draw (3) circle (2pt) node [below] {$\mathbb{C}$};
      \draw (4) circle (2pt) node [above] {$\mathbb{C}$};
      \draw[->, loop left, looseness = 40] (1) edge node [left] {0} (1);
      \draw[->, bend right = 20] (4) edge node [above left] {$a$} (1);
      \draw[->, bend left = 20] (4) edge node [right] {$b$} (1);
      \draw[->] (4) edge node [right] {$c$} (2);
      \draw[->] (4) edge node [above right] {$d$} (3);
      \draw[->, bend left = 20] (2) edge node [above] {$e$} (3);
      \draw[->, bend left = 20] (3) edge node [below] {$f$} (2);
    \end{tikzpicture}
  \end{equation}
  with~$ef=0$.
  We see that the resulting moduli space is isomorphic to~$\mathbb{P}^1\times(\mathbb{P}^1\vee\mathbb{P}^1)$:
  the first factor parametrizes~$(a:b)$,
  the second factor is the union of two copies of the thin sincere moduli space for the acyclic triangle quiver,
  indexed by which factor of~$ef$ vanishes,
  intersecting where~$e=f=0$.
\end{proof}

\begin{lemma}
  \label{lemma:third-type}
  For~$V\in S_{\xi_3}\subset\modulispace{2\vect{d}}\cong\mathbb{P}^2$,
  the fibre~$p^{-1}(V)$ is isomorphic to the product of two reduced singular conics.
\end{lemma}

\begin{proof}
  For the decomposition type~$\xi_3=(\vect{e}_K^1,\vect{e}_{\overline{K}}^1,\vect{e}_L^1,\vect{e}_{\overline{L}}^1)$
  we find the local quiver~$Q_{\xi_3}$
  with dimension vector~$\vect{i}_1+\vect{i}_2+\vect{i}_3+\vect{i}_4$
  and framing datum~$\vect{i}_1+\vect{i}_2+\vect{i}_3+\vect{i}_4$.
  We thus have to determine the moduli space of nilpotent semistable representations for the framing of the local quiver
  \begin{equation}
    \begin{tikzpicture}[node distance = 1.5cm]
      \node (1)                {};
      \node (2) [right of = 1] {};
      \node (3) [right of = 2] {};
      \node (4) [right of = 3] {};
      \node (0) at ($(2)!0.5!(3)+(0,1.2)$) {};
      \draw (0) circle (2pt) node [above] {0};
      \draw (1) circle (2pt) node [below] {1};
      \draw (2) circle (2pt) node [below] {2};
      \draw (3) circle (2pt) node [below] {3};
      \draw (4) circle (2pt) node [below] {4};
      \draw[->, bend left = 20] (1) edge (2);
      \draw[->, bend left = 20] (2) edge (1);
      \draw[->, bend left = 20] (3) edge (4);
      \draw[->, bend left = 20] (4) edge (3);
      \draw[->] (0) edge (1);
      \draw[->] (0) edge (2);
      \draw[->] (0) edge (3);
      \draw[->] (0) edge (4);
    \end{tikzpicture}
  \end{equation}
  for the dimension vector~$\vect{i}_0+\vect{i}_1+\vect{i}_2+\vect{i}_3+\vect{i}_4$ and stability~$4\vect{i}_0^*-\vect{i}_1^*-\vect{i}_2^*-\vect{i}_3^*-\vect{i}_4^*$. Similarly to the previous case, we find that the resulting moduli space is isomorphic to~$(\mathbb{P}^1\vee\mathbb{P}^1)^2$.
\end{proof}

\begin{lemma}
  \label{lemma:fourth-type}
  For~$V\in S_{\xi_4}\subset\modulispace{2\vect{d}}\cong\mathbb{P}^2$,
  the fibre~$p^{-1}(V)$ is isomorphic to the nodal quadric surface.
\end{lemma}

\begin{proof}
  For the decomposition type~$\xi_4=(\vect{d}^2)$
  we find the local quiver~$Q_{\xi_4}$
  with dimension vector~$2\vect{i}_1$ and framing datum~$2\vect{i}_1$.
  We thus have to determine the moduli space of nilpotent semistable representations of the framing of the local quiver,
  for dimension vector~$\vect{i}_0+2\vect{i}_1$ and stability~$2\vect{i}_0^*-\vect{i}_1^*$.
  Such a representation is of the form
  \begin{equation}
    \begin{tikzpicture}[node distance = 1.5cm]
      \node (1)                {$\mathbb{C}^2$};
      \node (2) [above of = 1] {$\mathbb{C}$};
      \draw[->, loop left, looseness = 10]  (1) edge node [left] {$A$} (1);
      \draw[->, bend right = 20] (2) edge node [left]  {$v$} (1);
      \draw[->, bend left  = 20] (2) edge node [right] {$w$} (1);
    \end{tikzpicture}
  \end{equation}
  where~$A$ is a nilpotent~$2\times 2$-matrix,
  and~$\mathbb{C}^2$ is generated by~$v,w,Av,Aw$.
  We can thus embed the moduli space into the Grassmannian~$\Gr(2,4)$ by taking the row space of the~$2\times 4$-matrix
  \begin{equation}
    [v\mid w \mid Av\mid Aw].
  \end{equation}
  Composing this embedding with the Pl\"ucker embedding, with maximal minors~$D_{ij}$,
  its image is defined by the equations~$D_{34}=0$,~$D_{14}=D_{23}$, and the Pl\"ucker relation~$D_{13}D_{24}=D_{12}D_{34}+D_{14}D_{23}$.
  Eliminating~$D_{34}$ and~$D_{23}$, we see that the moduli space is given by homogeneous coordinates~$(D_{12}:D_{13}:D_{14}:D_{24})$
  subject to~$D_{13}D_{24}=D_{14}^2$,
  thus it is isomorphic to the nodal quadric surface.
\end{proof}

\begin{lemma}
  \label{lemma:fifth-type}
  For~$V\in S_{\xi_5}\subset\modulispace{2\vect{d}}\cong\mathbb{P}^2$,
  the fibre~$p^{-1}(V)$ is isomorphic to the union of two copies of the second Hirzebruch surface~$\mathbb{F}_2$,
  with each~$(-2)$-curve glued to a fibre of the other surface.
\end{lemma}

\begin{proof}
  For the decomposition type~$\xi_5=(\vect{e}_K^2,\vect{e}_{\overline{K}}^2)$
  we find the local quiver~$Q_{\xi_5}$
  with dimension vector~$2(\vect{i}_1+\vect{i}_2)$ and framing datum~$\vect{i}_1+\vect{i}_2$.
  We thus have to determine the moduli space of nilpotent semistable representations of the framing of the local quiver,
  for dimension vector~$\vect{i}_0+2(\vect{i}_1+\vect{i}_2)$ and stability~$4\vect{i}_0^*-\vect{i}_1^*-\vect{i}_2^*$.
  Such a representation is given as follows:
  \begin{equation}
    \begin{tikzpicture}[node distance = 2.5cm]
      \node (a) {$\mathbb{C}^2$};
      \node (b) [right of = a] {$\mathbb{C}^2$};
      \node (c) at ($(a)!0.5!(b) + (0,1.1cm)$) {$\mathbb{C}$};
      \draw[->, bend left = 20] (a) edge node [above] {$A$} (b);
      \draw[->, bend left = 20] (b) edge node [below] {$B$} (a);
      \draw[->] (c) edge node [left]  {$v$} (a);
      \draw[->] (c) edge node [right] {$w$} (b);
    \end{tikzpicture}
  \end{equation}
  for vectors~$v,w\in\mathbb{C}^2$ and~$2\times 2$-matrices~$A,B$ such that~$AB$ is nilpotent,
  and such that the~$2\times 4$-matrices
  \begin{equation}
    D=[v\mid Bw\mid BAv\mid BABw],\; E=[w\mid Av\mid ABw\mid ABAv]
  \end{equation}
  have rank two.
  We can thus embed the moduli space into~$\Gr(2,4)\times\Gr(2,4)$,
  similar to \cref{lemma:fourth-type}, and we have to determine equations
  on the respective Pl\"ucker coordinates~$D_{ij},E_{ij}$ for~$1\leq i<j\leq 4$ describing the image of the embedding.
  First, we see that~$D_{34}=0=E_{34}$ since~$AB$ and~$BA$ are nilpotent.
  Next, the following are verified by direct computations:
  \begin{equation}
    D_{23}=D_{14},\; D_{14}^2=D_{13}D_{24},\; D_{14}=\det(B)E_{12},\; D_{24}=\det(B)E_{13},
  \end{equation}
  and
  \begin{equation}
    E_{23}=E_{14},\; E_{14}^2=E_{13}E_{24},\; E_{14}=\det(A)D_{12},\; E_{24}=\det(A)D_{13}.
  \end{equation}
  We can thus eliminate the variables~$D_{23}, D_{34}, E_{23}, E_{34}$.
  Since~$\det(A)\det(B)=0$, we find the following relations, which in fact define the image of the embedding:
  \begin{equation}
    \begin{gathered}
      D_{14}E_{14}=D_{14}E_{24}=D_{24}E_{14}=D_{24}E_{24}=0, \\
      D_{13}E_{14}=D_{12}E_{24},\; D_{14}E_{13}=D_{24}E_{12}, \\
      D_{13}E_{13}=D_{12}E_{14}+D_{14}E_{12}.
    \end{gathered}
  \end{equation}
  We see that either~$D_{14}=0=D_{24}$ or~$E_{14}=0=E_{24}$,
  thus the fibre has two irreducible components~$S_1,S_2$,
  intersecting in the locus~$C'$ where all four variables are zero,
  which consists of two copies of a projective line intersecting in a single point,
  that is,~$C'=\mathbb{P}^1\vee\mathbb{P}^1$.
  Considering the component~$S_1$ defined by~$D_{14}=0=D_{24}$,
  we see as before that the variables~$E_{ij}$ define the nodal quadric surface.
  The additional relations involving~$D_{12}, D_{13}$ show that~$S_1$ is the blow-up of this cone at the vertex,
  and thus isomorphic to the second Hirzebruch surface~$\mathbb{F}_2$, and similarly for~$S_2$,
  and we see that each~$(-2)$-curve is glued to a fibre of the other surface.
\end{proof}

\subsection{Involution surface bundle structure}
From the discussion in \cref{subsection:fibre-types}
we see that the morphism~$p\colon\QM{1}\to\mathbb{P}^2$
looks generically like a quadric surface bundle,
but with singular fibres which are not merely singular quadric surfaces.
Recently, Kresch--Tschinkel studied such fibrations,
we recall their \cite[Definition~1]{MR4159824}.
\begin{definition}
  \label{definition:involution-surface-bundle}
  Let~$S$ be a regular scheme,
  for which~2 is invertible in the local rings.
  An \emph{involution surface bundle} is
  a flat projective morphism~$p\colon X\to S$ for which
  \begin{itemize}
    \item the locus~$U\subseteq S$ over which~$p$ is smooth is Zariski-dense,
    \item the fibre~$p^{-1}(s)$ for~$s\in U$ is an involution surface,
      i.e., a minimal del Pezzo surface of degree~8.
  \end{itemize}
  We say it has \emph{mild degeneration} if every singular fibre is geometrically isomorphic to
  one of the following reduced schemes:
  \begin{description}
    \item[type I]
      a nodal quadric surface
    \item[type II]
      the product of two reduced singular conics
    \item[type III]
      the union of two copies of the second Hirzebruch surface~$\mathbb{F}_2$,
      with each~$(-2)$-curve glued to a fibre of the other surface
    \item[type IV]
      the product of a smooth conic with a reduced singular conic.
  \end{description}
\end{definition}
We will prove the following theorem.
\begin{theorem}
  \label{theorem:involution-surface-bundle}
  The morphism~$p\colon\QM{1}\to\mathbb{P}^2$ from \eqref{equation:involution-surface-bundle}
  is an involution surface bundle which has mild degeneration.
  It degenerates over a configuration of three lines
  and a smooth conic tangent to each of them,
  depicted in \cref{figure:discriminant-locus}:
  \begin{itemize}
    \item over the three lines, the (general) fibre is of type~IV, and
    \item over the conic, the (general) fibre is of type~I;
    \item over the three vertices of the triangle formed by the lines,
      the fibre is of type~II, and
    \item over the three points where the conic touches the lines,
      the fibre is of type~III.
  \end{itemize}
\end{theorem}

\begin{proof}
  The morphism~$p$ is projective and flat,
  by the discussion following \cref{lemma:qm-1-projection},
  and by \cref{lemma:first-fiber} its general fibre is a smooth quadric surface,
  which is an involution surface.
  It remains to describe the Luna stratification of~$\modulispace{2\vect{d}}\cong\mathbb{P}^2$.
  Under the isomorphism~$\modulispace{2\vect{d}}\cong\Sym^2\modulispace{\vect{d}}$,
  a polystable representation corresponds to an unordered pair of points
  of~$\modulispace{\vect{d}}\cong\mathbb{P}^1$,
  where the three points~$0,1,\infty$, at which the cross-ratio degenerates,
  correspond to the polystable representations~$\vect{e}_K\oplus\vect{e}_{\overline{K}}$.
  The stratum~$S_{\xi_2}$ consists of the pairs containing one fixed degenerate point,
  so its closure is a line in~$\mathbb{P}^2$, and we obtain three such lines;
  the three points of the strata of type~$\xi_3$,
  the pairs of two distinct degenerate points,
  are the vertices of the triangle formed by these lines.
  The stratum~$S_{\xi_4}$ is the diagonal,
  which gives a smooth conic.
  This conic meets the line of pairs containing a fixed degenerate point
  in the single point given by the pair of type~$\xi_5$,
  so it is tangent to each of the three lines,
  with the three tangency points being the strata of type~$\xi_5$.
  The fibre types over these strata are determined in
  \cref{lemma:second-type,lemma:third-type,lemma:fourth-type,lemma:fifth-type}.
\end{proof}

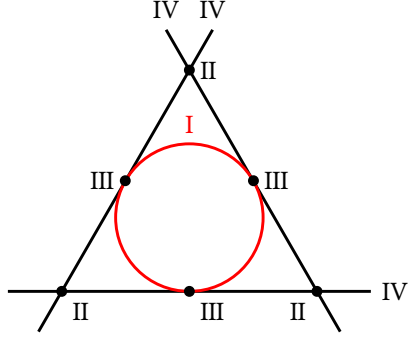
\begin{figure}[t]
  \centering
  \begin{tikzpicture}[scale=0.75, font=\small, line width=1.1pt]
    \draw (-3.2,-1.3) -- (3.2,-1.3)    node[right] {IV};
    \draw (2.66,-2.01) -- (-0.41,3.31) node[above] {IV};
    \draw (-2.66,-2.01) -- (0.41,3.31) node[above] {IV};
    \draw[red] (0,0) circle (1.3);
    \node[red] at (0,1.65) {I};
    \filldraw (0,2.6)      circle (2pt) node[right]       {II};
    \filldraw (-2.25,-1.3) circle (2pt) node[below right] {II};
    \filldraw (2.25,-1.3)  circle (2pt) node[below left]  {II};
    \filldraw (0,-1.3)     circle (2pt) node[below right] {III};
    \filldraw (1.13,0.65)  circle (2pt) node[right]       {III};
    \filldraw (-1.13,0.65) circle (2pt) node[left]        {III};
  \end{tikzpicture}
  \caption{The discriminant locus of the involution surface bundle~$p\colon\QM{1}\to\mathbb{P}^2$
  from \eqref{equation:involution-surface-bundle}, described in \cref{theorem:involution-surface-bundle};
  the labels indicate the fibre type, in the sense of \cref{definition:involution-surface-bundle},
  over the general point of each stratum.}
  \label{figure:discriminant-locus}
\end{figure}

\section{The variety \QM{2}: Segre cousin}
\label{section:QM2}
This moduli space has Betti numbers~$(1,6,17,6,1)$;
our aim is to prove in \cref{subsection:segre-cousin}
that it is isomorphic to Manivel's 4-dimensional Segre cousin \cite{MR4739832}.

\subsection{Projections to the wall}
Besides the canonical stability parameter
\begin{equation}
  \Theta=\Theta_\can=3\sum_{k=1}^6\vect{i}_k^*-7\vect{j}^*,
\end{equation}
we also consider
\begin{equation}
  \begin{aligned}
    \overline{\Theta}_1&=\sum_{k=1}^5\vect{i}_k^*+2\vect{i}_6^*-3\vect{j}^* \\
    \overline{\Theta}_2&=\sum_{k=1}^52\vect{i}_k^*+\vect{i}_6^*-4\vect{j}^*.
  \end{aligned}
\end{equation}
We will apply the tools from \cref{subsection:projection-to-wall} to obtain two morphisms between quiver moduli.
\begin{lemma}
  The stability parameter~$\Theta$ does not belong to a proper wall.
  Moreover,~$\overline{\Theta}_1$ and~$\overline{\Theta}_2$ belong to the closure of the chamber containing~$\Theta$.
\end{lemma}

\begin{proof}
  Consider a non-zero proper dimension vector~$\vect{e}\leq\vect{d}=(1^5,2;3)$ such that~$\Theta(\vect{e})\leq 0$.
  Then
  \begin{equation}
    \vect{e}=\sum_{k\in K}\vect{i}_k+e\vect{i}_6+d\vect{j}
  \end{equation}
  for a subset~$K\subset\{1,\ldots,5\}$ of cardinality~$f$ such that
  \begin{equation}
    (0,0,0)\lneqq(f,e,d)\lneqq(5,2,3),
  \end{equation}
  and we have~$3(f+e)\leq 7d$, and thus strict inequality:
  indeed,~$3$ divides~$3(f+e)$ but not~$7d$ unless~$d=3$,
  in which case equality would force~$(f,e,d)=(5,2,3)=\vect{d}$.
  This proves that~$\Theta$ does not belong to a proper wall.

  It also follows from this condition, by a quick case-by-case consideration,
  that~$2f+e\leq 4d$, thus~$\overline{\Theta}_2$ belongs to the closure of the chamber containing~$\Theta$.
  If~$f+2e>3d$, we find, again by a case-by-case consideration,
  that~$(f,e,d)=(0,2,1)$.
  But a representation of dimension vector~$2\vect{i}_6+\vect{j}$
  certainly contains a subrepresentation of dimension vector~$\vect{i}_6$.

  By~\cref{lemma:projection}, we see that also~$\overline{\Theta}_1$ belongs to the closure of the chamber containing~$\Theta$.
\end{proof}
\subsection{First projection}
\label{subsection:QM2-first-projection}
We consider the projection
\begin{equation}
  \label{equation:p1-qm-2}
  p_1\colon\QM{2}\to\modulispace[\overline{\Theta}_1\semistable]{Q^{(6)},\vect{d}}.
\end{equation}
We determine its fibres and the singularities of the target.

The proper non-zero dimension vectors~$\vect{e}\leq\vect{d}$
for which~$\overline{\Theta}_1(\vect{e})=0$ holds are
\begin{itemize}
  \item $\vect{e}^1_k=\vect{i}_k+\vect{i}_6+\vect{j}$, for~$k=1,\ldots,5$;
  \item $\vect{e}^2_{kl}=\sum_{p\not=k,l}\vect{i}_p+\vect{j}$ for~$1\leq k<l\leq 5$;
  \item $\vect{e}^3_k=\sum_{l\neq k}\vect{i}_l+\vect{i}_6+2\vect{j}$ for~$k=1,\ldots,5$;
  \item $\vect{e}^1_k+\vect{e}^1_l$ for~$k\neq l$.
\end{itemize}
For this latter dimension vector there are no~$\overline{\Theta}_1$-stable representations,
since the Euler pairing with itself is~$2$.
For the other three dimension vectors
there are stable representations,
as direct inspection of a general representation easily shows
(or alternatively one can use~\cref{lemma:stable-general-position}).

We can thus enumerate the decomposition types, besides~$(\vect{d}^1)$, as
\begin{enumerate}
  \item $\xi_1=((\vect{e}^1_k)^1,(\vect{e}^3_k)^1)$ for~$k=1,\ldots,5$, and
  \item $\xi_2=((\vect{e}^1_k)^1,(\vect{e}^1_l)^1,(\vect{e}^2_{kl})^1)$ for~$1\leq k<l\leq 5$.
\end{enumerate}
Recall that by~\cref{proposition:fibres-projection-wall}, we have
\begin{equation}
  p_1^{-1}(V)\cong\modulispace[\Theta_\xi\semistable,\nilpotent]{Q_\xi,\vect{d}_\xi},
\end{equation}
while by \cref{proposition:local-quiver}
the singularity of a point in the Luna stratum corresponding to a decomposition type~$\xi$
is \'etale equivalent to the singularity of~$0$ in~$\modulispace[0\semistable]{Q_\xi,\vect{d}_\xi}$.
For the trivial Luna type~$(\vect{d}^1)$
the fibre is a point
and the points of the stratum are smooth.

\begin{lemma}
  \label{lemma:fiber-type-qm-2-1}
  For~$V\in S_{\xi_1}\subset\modulispace[\overline{\Theta}_1\semistable]{Q^{(6)},\vect{d}}$,
  the fibre~$p_1^{-1}(V)$ is isomorphic to~$\mathbb{P}^1$.
\end{lemma}

\begin{proof}
  For the decomposition type~$\xi_1$, the local quiver~$Q_{\xi_1}$
  has~$\vect{d}_{\xi_1}=\vect{i}_1+\vect{i}_2$ and~$\Theta_{\xi_1}=-\vect{i}_1^*+\vect{i}_2^*$,
  and a nilpotent representation of~$Q_{\xi_1}$ of dimension vector~$\vect{d}_{\xi_1}$
  is of the form
  \begin{equation}\label{quiver122}
    \begin{tikzpicture}[node distance = 1.5cm]
      \node (1)                {};
      \node (2) [right of = 1] {};
      \draw (1) circle (2pt) node [above] {$\mathbb{C}$};
      \draw (2) circle (2pt) node [above] {$\mathbb{C}$};
      \draw[->, bend left = 50] (1) edge node [above] {$a$} (2);
      \draw[->, bend left = 30] (2) edge node [above] {$b_1$} (1);
      \draw[->, bend left = 60] (2) edge node [below] {$b_2$} (1);
      \draw[->, loop below] (2) edge node [right] {0} (2);
      \draw[->, loop right] (2) edge node [below] {0}(2);
    \end{tikzpicture}
  \end{equation}
  By the nilpotency we obtain that~$ab_1=ab_2=0$.

  A representation is~$\Theta_{\xi_1}$-semistable if and only if~$(b_1,b_2)\neq(0,0)$:
  the only potentially destabilising subrepresentation is the 1-dimensional subrepresentation at the vertex~$2$,
  of dimension vector~$\vect{i}_2$, and~$\Theta_{\xi_1}(\vect{i}_2)=1>0$,
  so it destabilises precisely when~$b_1=b_2=0$.
  On the semistable nilpotent locus we therefore have~$a=0$ and~$(b_1,b_2)\neq(0,0)$,
  so taking the quotient gives~$\mathbb{P}^1$.
\end{proof}

Since~$\overline{\Theta}_1$ lies on a wall of~$\Stab(\vect{d})$,
the moduli space~$\modulispace[\overline{\Theta}_1\semistable]{Q^{(6)},\vect{d}}$ need not be smooth a~priori.
However,
we have the following.

\begin{lemma}
  \label{lemma:stratum-qm-2-1}
  At every point of the Luna stratum~$S_{\xi_1}$
  the moduli space~$\modulispace[\overline{\Theta}_1\semistable]{Q^{(6)},\vect{d}}$ is smooth.
\end{lemma}

\begin{proof}
  Let~$(a,b_1,b_2,\ell_1,\ell_2)\in\mathbb{A}^5$
  correspond to a representation as in \eqref{quiver122},
  without the nilpotency condition,
  so~$\ell_1$ and~$\ell_2$ are scalars associated to the loops.
  For the trivial stability parameter every representation is semistable.
  The group~$\Gm^2$ acts on an arrow coordinate by the ratio of the scalars at its target and source,
  and the diagonal scalar acts trivially.
  Thus the effective base change group is~$\Gm$;
  writing its parameter as~$\lambda=g_2/g_1$,
  it acts with weight~$1$ on~$a$,
  weight~$-1$ on~$b_1,b_2$,
  and trivially on~$\ell_1,\ell_2$.
  Hence the invariant ring is the polynomial ring
  \begin{equation}
    \mathbb{C}[\ell_1,\ell_2,ab_1,ab_2],
  \end{equation}
  so~$\modulispace[0\semistable]{Q_{\xi_1},\vect{d}_{\xi_1}}\cong\mathbb{A}^4$ is smooth.
  The claim now follows from~\cref{proposition:local-quiver}.
\end{proof}

\begin{lemma}
  \label{lemma:fiber-type-qm-2-2}
  For~$V\in S_{\xi_2}\subset\modulispace[\overline{\Theta}_1\semistable]{Q^{(6)},\vect{d}}$,
  the fibre~$p_1^{-1}(V)$ is isomorphic to~$\mathbb{P}^1\times\mathbb{P}^1$.
\end{lemma}

\begin{proof}
  For the decomposition type~$\xi_2$, the local quiver~$Q_{\xi_2}$
  has~$\vect{d}_{\xi_2}=\vect{i}_1+\vect{i}_2+\vect{i}_3$ and~$\Theta_{\xi_2}=-\vect{i}_1^*+2\vect{i}_2^*-\vect{i}_3^*$,
  and a nilpotent representation of~$Q_{\xi_2}$ of dimension vector~$\vect{d}_{\xi_2}$
  is of the form
  \begin{equation}\label{quiver1212}
    \begin{tikzpicture}[node distance = 1.5cm]
      \node (1)                {};
      \node (2) [right of = 1] {};
      \node (3) [right of = 2] {};
      \draw (1) circle (2pt) node [above] {$\mathbb{C}$};
      \draw (2) circle (2pt) node [above] {$\mathbb{C}$};
      \draw (3) circle (2pt) node [above] {$\mathbb{C}$};
      \draw[->, bend left = 50] (1) edge node [above] {$a$}  (2);
      \draw[->, bend left = 30] (2) edge node [above] {$b_1$} (1);
      \draw[->, bend left = 60] (2) edge node [below] {$b_2$} (1);
      \draw[->, bend right = 50] (3) edge node [above] {$c$}  (2);
      \draw[->, bend right = 30] (2) edge node [above] {$d_1$} (3);
      \draw[->, bend right = 60] (2) edge node [below] {$d_2$} (3);
    \end{tikzpicture}
  \end{equation}
  By the nilpotency we obtain that~$ab_1=ab_2=0$ and~$cd_1=cd_2=0$.
  The description of the fibre is similar to the previous case,
  but now we have two independent pairs of arrows,
  and we obtain that the fibre is isomorphic to~$\mathbb{P}^1\times\mathbb{P}^1$.
\end{proof}

Smoothness also holds along the second type of Luna stratum.

\begin{lemma}
  \label{lemma:stratum-qm-2-2}
  At every point of the Luna stratum~$S_{\xi_2}$
  the moduli space~$\modulispace[\overline{\Theta}_1\semistable]{Q^{(6)},\vect{d}}$ is smooth.
\end{lemma}

\begin{proof}
  As in the proof of \cref{lemma:stratum-qm-2-1},
  let~$(a,b_1,b_2,c,d_1,d_2)\in\mathbb{A}^6$
  correspond to a representation as in \eqref{quiver1212},
  without the nilpotency condition.
  The effective base change group is now~$\Gm^2$,
  with~$\lambda=g_2/g_1$ acting with weight~$1$ on~$a$ and weight~$-1$ on~$b_1,b_2$,
  and~$\mu=g_2/g_3$ acting with weight~$1$ on~$c$ and weight~$-1$ on~$d_1,d_2$.
  Hence the invariant ring is the polynomial ring~$\mathbb{C}[ab_1,ab_2,cd_1,cd_2]$,
  so~$\modulispace[0\semistable]{Q_{\xi_2},\vect{d}_{\xi_2}}\cong\mathbb{A}^4$ is smooth,
  and the claim follows from~\cref{proposition:local-quiver}.
\end{proof}

We have thus proved the following.

\begin{proposition}
  The projection~$p_1$ from \eqref{equation:p1-qm-2}
  is birational with smooth target,
  with fibre over each point of~$S_{\xi_1}$ isomorphic to~$\mathbb{P}^1$
  and fibre over each point of~$S_{\xi_2}$ isomorphic to~$\mathbb{P}^1\times\mathbb{P}^1$.
\end{proposition}

\begin{proof}
  Over the open Luna stratum~$(\vect{d}^1)$,
  the fibre is a point,
  so~$p_1$ is birational.
  By~\cref{lemma:fiber-type-qm-2-1,lemma:fiber-type-qm-2-2},
  the fibres over the proper Luna strata~$S_{\xi_1}$ and~$S_{\xi_2}$
  are respectively isomorphic to~$\mathbb{P}^1$ and~$\mathbb{P}^1\times\mathbb{P}^1$.
  The open stratum is smooth,
  since it consists of stable representations of the acyclic quiver~$Q^{(6)}$,
  and the proper strata are smooth by~\cref{lemma:stratum-qm-2-1,lemma:stratum-qm-2-2}.
\end{proof}

\paragraph{The base of the projection}
We now describe the moduli space~$\modulispace[\overline{\Theta}_1\semistable]{Q^{(6)},\vect{d}}$
(and its Luna strata) in terms of Grassmannians using explicit coordinates.
We consider a representation of~$Q^{(6)}$ of dimension vector~$\vect{d}$ as a tuple
\begin{equation}
  (v_1,\ldots,v_5,A)
\end{equation}
of vectors~$v_i\in\mathbb{C}^3$ together with a~$3\times 2$-matrix~$A$,
which we represent as the~$3\times 7$-matrix
\begin{equation}
  M=[v_1\mid v_2\mid v_3\mid v_4\mid v_5\mid A].
\end{equation}
We can thus define semi-invariant functions on~$\repspace{Q^{(6)},\vect{d}}$ using minors of~$M$.
Namely, we consider
\begin{equation}
  D_{ij}\colonequals\det[v_p\mid v_q\mid v_r]
\end{equation}
for~$1\leq i,j\leq 5$,~$i\not=j$ and~$\{1,\ldots,5\}\setminus\{i,j\}=\{p<q<r\}$, as well as
\begin{equation}
  E_i\colonequals\det[v_i\mid A]
\end{equation}
for~$i=1,\ldots,5$.
We then find semi-invariants
\begin{equation}
  \label{equation:Sij}
  S_{ij}=D_{ij}E_iE_j,
\end{equation}
for~$1\leq i<j\leq 5$, of weight~$\overline{\Theta}_1$.
\begin{lemma}
  The map from~$\repspace[\overline{\Theta}_1\semistable]{Q^{(6)},\vect{d}}$ to~$\mathbb{P}^9$
  given by the coordinate functions~$S_{ij}$ is well-defined.
\end{lemma}

\begin{proof}
  Suppose, to the contrary, that all~$S_{ij}$ evaluate to zero on a matrix~$M$.
  For a~$\overline{\Theta}_1$-semistable representation, all arrows are represented by injections,
  thus all vectors~$v_i$ are non-zero,
  and the matrix~$A$ has rank~$2$;
  let~$H\subset \mathbb{C}^3$ be the span of its columns.

  The vectors~$v_1,\ldots,v_5$ span~$\mathbb{C}^3$: otherwise their span is contained in a hyperplane~$H'$,
  which intersects~$H$ in a non-zero vector~$v$.
  This induces a subrepresentation of dimension vector~$\sum_{k=1}^5\vect{i}_k+\vect{i}_6+2\vect{j}$, a contradiction.
  Thus there exists a three-element subset~$K\subset\{1,\ldots,5\}$
  such that~$(v_k)_{k\in K}$ forms a basis of~$\mathbb{C}^3$.

  By assumption,~$v_l\in H$ for an index~$l\not\in K$.
  At most two of the vectors~$v_1,\ldots,v_5$ are contained in~$H$:
  three such vectors~$v_p,v_q,v_r\in H$
  induce a subrepresentation of dimension vector~$\vect{i}_p+\vect{i}_q+\vect{i}_r+2\vect{i}_6+2\vect{j}$, a contradiction.
  If both~$v_l,v_{l'}$ for~$K\cup\{l,l'\}=\{1,\ldots,5\}$ are contained in~$H$,
  then~$v_k,v_l,v_{l'}$ for any~$k\in K$ are linearly dependent,
  thus their span is contained in~$H$.
  This induces a subrepresentation of dimension vector~$\vect{i}_k+\vect{i}_l+\vect{i}_{l'}+2\vect{i}_6+2\vect{j}$, a contradiction.
  Thus~$v_l$ is the only vector contained in~$H$.
  But then~$v_l$ is linearly dependent with any two of the three vectors~$(v_k)_{k\in K}$,
  again a contradiction.
\end{proof}

The above coordinates~$S_{ij}$ satisfy the Pl\"ucker relations
\begin{equation}
  S_{ik}S_{jl}=S_{ij}S_{kl}+S_{il}S_{jk}
\end{equation}
for~$1\leq i<j<k<l\leq 5$, as well as the linear relations
\begin{equation}
  \sum_{j<i}(-1)^{j}S_{ji}+\sum_{j>i}(-1)^{j-1}S_{ij}=0
\end{equation}
for all~$i=1,\ldots,5$. We can thus eliminate the~$S_{i,5}$, say, and end up with coordinates~$S_{i,j}$ for~$1\leq i<j\leq 4$ satisfying the defining relations of~$\Gr(2,4)$.

\begin{lemma}
  \label{lemma:QM2-target-grassmannian}
  The morphism~$p\colon\modulispace[\overline{\Theta}_1\semistable]{Q^{(6)},\vect{d}}\rightarrow\Gr(2,4)$ is an isomorphism,
  mapping the union of the proper Luna strata to a configuration of five projective planes meeting in ten points.
\end{lemma}

\begin{proof}
  Since the target is normal,
  it suffices by Zariski's main theorem to prove that
  the proper birational morphism~$p$ is bijective,
  which we do stratum by stratum.
  So assume that a representation, given by~$(v_1,\ldots,v_5,A)$ as above, is~$\overline{\Theta}_1$-stable.
  If some~$v_k$ belongs to~$H$,
  we find a subrepresentation of dimension vector~$\vect{e}^1_k$, a contradiction.
  We can thus normalize the matrix~$M$ to the form
  \begin{equation}
    M=\left[
      \begin{array}{ccccccc}1&1&1&1&1&0&0\\ a_{11}&a_{12}&a_{13}&a_{14}&0&1&0\\ a_{21}&a_{22}&a_{23}&a_{24}&0&0&1
    \end{array}\right].
  \end{equation}
  The submatrix~$(a_{ik})_{i,k}$ has rank two since we already know that~$v_1,\ldots,v_5$ span~$\mathbb{C}^3$,
  and the coordinates~$S_{kl}$ for~$1\leq k<l\leq 4$ precisely evaluate to
  the Pl\"ucker coordinates of the~$2\times 4$-matrix~$(a_{ik})_{i,k}$,
  proving bijectivity of the restriction of~$p$ to the open Luna stratum.

  Assume now that a representation belongs to the Luna stratum~$((\vect{e}_5^1)^1,(\vect{e}_5^3)^1)$,
  say, thus the matrix~$M$ is of the form
  \begin{equation}
    M=\left[
      \begin{array}{ccccccc}0&0&0&0&1&1&0\\ w_1&w_2&w_3&w_4&0&0&w
    \end{array}\right]
  \end{equation}
  for~$w_1,w_2,w_3,w_4,w\in\mathbb{C}^2$.
  Since these vectors define a~$\overline{\Theta}_1$-stable representation of dimension vector~$\vect{e}_1^3$,
  we see that none of the~$w_k$ are collinear to~$w$,
  since this would yield a subrepresentation of dimension vector~$\vect{e}_k^1$. Thus we can normalize~$M$ to
  \begin{equation}
    M=\left[
      \begin{array}{ccccccc}0&0&0&0&1&1&0\\ a_1&a_2&a_3&0&0&0&1\\ 1&1&1&1&0&0&0
    \end{array}\right].
  \end{equation}
  The coordinates~$S_{kl}$ allow us to read off~$(a_1:a_2:a_3)$,
  showing bijectivity of the restriction of~$p$ to this Luna stratum,
  and showing that its closure is isomorphic to~$\mathbb{P}^2$.

  This proves that the map~$p$ is an isomorphism,
  and the image of the union of proper Luna strata is a configuration of five planes~$P_k$ meeting in ten points~$q_{kl}$,
  that is,~$P_k\cap P_l=\{q_{kl}\}$ for~$1\leq k<l\leq 5$.
\end{proof}
We thus conclude:

\begin{theorem}
  \label{theorem:QM2-grassmannian}
  The projection~$p_1\colon\QM{2}\rightarrow\Gr(2,4)$ is birational and semismall,
  with general fibre over the five planes~$P_k$ isomorphic to~$\mathbb{P}^1$,
  and fibre over the ten points~$q_{kl}$ isomorphic to~$\mathbb{P}^1\times\mathbb{P}^1$.
\end{theorem}
Indeed, the exceptional strata have dimensions~$2$ and~$0$ respectively,
and the corresponding fibres have dimensions~$1$ and~$2$.

In the next subsection we use this geometric description
to show that~$\QM{2}$ is isomorphic to Manivel's Segre cousin.

\subsection{Identification with the Segre cousin}
\label{subsection:segre-cousin}
Let us first recall from \cite{MR4739832}
the construction of the Segre cousin and the properties we will use.
Let~$V_4$ and~$V_5$ be vector spaces of dimension~4 and~5.
The \emph{Segre cousin} is the smooth Fano fourfold~$X$
obtained as the zero locus of a general section~$\theta$ of the vector bundle
\begin{equation}
  \mathcal{U}^\vee\boxtimes\wedge^2\mathcal{V}^\vee
\end{equation}
on~$\Gr(2,V_4)\times\Gr(3,V_5)$,
where~$\mathcal{U}$ and~$\mathcal{V}$ denote the tautological subbundles of rank~2 and~3.
It first arose in the database \cite{fanofourfolds}
as the unique Fano fourfold of maximal Picard rank (namely~6) therein.
By \cite[Proposition~4.1]{MR4739832}
it is rational and of index~1,
with~$\mathrm{h}^0(X,\omega_X^\vee)=40$ and~$\mathrm{c}_1^4=172$,
and its cohomology is of pure Hodge--Tate type
with~$\mathrm{h}^{1,1}=6$ and~$\mathrm{h}^{2,2}=17$,
all in agreement with the invariants of~$\QM{2}$ in \cref{table:overview}.

The section~$\theta$ is a general element of~$V_4^\vee\otimes\wedge^2V_5^\vee$,
i.e., a general four-dimensional space of skew-symmetric forms in five variables.
This space contains exactly five rank-2 forms up to scalars,
which determine five points in~$\mathbb{P}(V_4)$,
five planes in~$\mathbb{P}(V_5)$,
and five planes~$\pi_1,\ldots,\pi_5$ in~$\Gr(2,V_4)$ belonging to the same ruling,
all in general position \cite[\S2]{MR4739832}.
The relation to the Segre cubic primal is through the second projection:
by \cite[Proposition~4.2]{MR4739832}
the morphism~$\rho_2\colon X\to\Gr(3,V_5)$ is a small resolution of
a singular fourfold with ten singular points,
contracting ten planes which form
part of a Cremona--Richmond configuration,
in analogy with the small resolutions of the Segre cubic itself.
Strictly speaking, it is therefore this singular fourfold
which is the 4-dimensional analogue of the (itself singular) Segre cubic,
with Manivel's Segre cousin a small resolution of it.

For our purposes the first projection is the relevant one:
by \cite[Proposition~4.6]{MR4739832},
the morphism
\begin{equation}
  \rho_1\colon X\to\Gr(2,V_4)
\end{equation}
is birational,
its exceptional locus lies over the union of the five planes~$\pi_1,\ldots,\pi_5$,
which intersect pairwise in ten points,
the fibre over a general point of each plane is a~$\mathbb{P}^1$,
and the fibre over each of the ten intersection points is~$\mathbb{P}^1\times\mathbb{P}^1$.
This yields the blow-up-and-contract description of~$X$
given in the introduction of \cite{MR4739832} (see also \S4.4 of op.~cit.):
blowing up~$\Gr(2,4)$ first in the ten points
and then in the strict transforms of the five planes,
the strict transforms of the exceptional divisors of the first blow-up
can be contracted to the smooth Fano fourfold~$X$.

We will prove the following.
\begin{theorem}
  \label{theorem:segre-cousin}
  The quiver moduli space~$\QM{2}$
  is isomorphic to the Segre cousin~$X$.
\end{theorem}

The proof compares the two birational semismall morphisms to~$\Gr(2,4)$;
an explicit moduli-theoretic isomorphism would be desirable,
but we have not found one.

The following result is the key input we need from birational geometry.
In the literature it is also referred to as the uniqueness of ample models,
see \cite[Lemma~3.6.6(1)]{MR2601039},
although the result certainly predates that reference.
\begin{proposition}
  \label{proposition:lehn}
  Let~$f\colon X\dashrightarrow Y$ be
  a birational map of normal projective varieties,
  which is an isomorphism in codimension~1.
  Let~$L$ be an ample Cartier divisor on~$Y$,
  whose proper transform~$f^*L$ on~$X$ is again ample.
  Then~$f$ extends to an isomorphism~$X\cong Y$.
\end{proposition}

\paragraph{Two semismall models and five points}
By \cref{theorem:QM2-grassmannian},
resp.~by \cite[Proposition~4.6]{MR4739832} as recalled above,
the semismall birational morphisms~$p_1\colon\QM{2}\to\Gr(2,4)$
and~$\rho_1\colon X\to\Gr(2,4)$
are isomorphisms over the complement of a configuration of
five planes~$P_1,\ldots,P_5$ (resp.~$P_1',\ldots,P_5'$) in~$\Gr(2,4)$,
with fibre~$\mathbb{P}^1$ over a general point of each plane,
and fibre~$\mathbb{P}^1\times\mathbb{P}^1$ over
the ten points~$q_{k,\ell}=P_k\cap P_\ell$ (resp.~$q_{k,\ell}'=P_k'\cap P_\ell'$).
We will analyze the discriminants of these morphisms,
showing that they can be identified.

Using that~$\Gr(2,4)$ is both the 4-dimensional quadric,
and the moduli space of lines in~$\mathbb{P}^3$
we know that planes in~$\Gr(2,4)$ come in two families:
\begin{itemize}
  \item planes parameterized by~$\mathbb{P}^3$,
    given by the lines through a point~$p\in\mathbb{P}^3$;
  \item planes parameterized by~$\mathbb{P}^{3,\vee}$,
    given by the lines in a hyperplane of~$\mathbb{P}^3$.
\end{itemize}
Two planes in the same family meet in a point,
two planes from a different family either meet in a line
or they are disjoint.

\begin{lemma}
  \label{lemma:uniqueness-of-configuration}
  Let~$\{R_1,\ldots,R_5\}$ be
  five distinct planes in~$\Gr(2,4)$
  intersecting in~ten distinct points.
  Then
  \begin{enumerate}
    \item\label{item:same-ruling}
      they belong to the same ruling,
      so we can assume without loss of generality that they correspond to
      five distinct points~$r_1,\ldots,r_5$ in~$\mathbb{P}^3$;
    \item\label{item:not-collinear}
      no three of the~$r_1,\ldots,r_5$ are collinear;
    \item\label{item:not-coplanar}
      if no four of the~$r_1,\ldots,r_5$ are coplanar,
      then the configuration is unique up to~$\Aut(\Gr(2,4))$:
      any two such configurations are conjugate
      under the action of~$\PGL_4\subseteq\Aut(\Gr(2,4))$.
  \end{enumerate}
\end{lemma}

\begin{proof}
  \Cref{item:same-ruling} is evident from the geometry recalled before the statement,
  and using the duality on~$\Gr(2,4)$ exchanging points and planes.
  \Cref{item:not-collinear} follows from the assumption that the ten points are distinct.
  Finally, \Cref{item:not-coplanar}
  follows because the non-coplanarity assumption forces the five points
  to form a projective frame of~$\mathbb{P}^3$,
  and~$\PGL_4$ acts simply transitively on the set of ordered frames,
  showing how any two such sets of points are projectively equivalent.
\end{proof}

We want to establish that both discriminants
are of the form in \Cref{item:not-coplanar}
in \cref{lemma:uniqueness-of-configuration}
to conclude that they are the same.
Note that \cref{item:not-collinear} alone does not suffice:
for~$r_1,\ldots,r_4$ coplanar in general position,
and~$r_5$ outside the plane they span,
the cross-ratio of the coplanar points is a continuous parameter.

We will apply the following recognition lemma.
\begin{lemma}
  \label{lemma:recognition}
  Let~$\{r_1,\ldots,r_5\}$ be five distinct points in~$\mathbb{P}^3$,
  and assume no three are collinear.
  Consider the action of~$\mathrm{S}_5$ on~$\mathbb{P}^3$,
  seen as a subgroup of~$\PGL_4$,
  and assume that~$\mathrm{S}_5\cdot\{r_1,\ldots,r_5\}=\{r_1,\ldots,r_5\}$
  such that~$\sigma\cdot r_i=r_{\sigma(i)}$ for all~$\sigma\in\mathrm{S}_5$
  and~$i=1,\ldots,5$.
  Then no four points are coplanar.
  In particular,
  they form a projective frame.
\end{lemma}

\begin{proof}
  Let~$\mathbb{P}^k\subseteq\mathbb{P}^3$ be
  the linear span of the~$\mathrm{S}_5$-orbit.
  Then~$k\geq 2$,
  because we assumed the points are not collinear.
  If~$k=2$,
  then~$\mathrm{S}_5$ would be a subgroup of~$\PGL_3$,
  which is not possible,
  hence~$k=3$.
  Next, the group~$\mathrm{S}_5$
  permutes the five~4-element subsets transitively,
  so either all such 4-element subsets are coplanar,
  or none are.
  But if all are, we must have~$k=2$, which is not possible.
\end{proof}

We have that both discriminants give rise to a projective frame:
\begin{enumerate}
  \item The group~$\mathrm{S}_5$
    permutes the vertices~$i_1,\ldots,i_5\in Q_0^{(6)}$
    (whilst fixing~$i_6$ and~$j$),
    inducing automorphisms of~$\QM{2}$,
    and descending to an automorphism of~$\Gr(2,4)$.
    By \cref{lemma:QM2-target-grassmannian}
    the plane~$P_k\subset\Gr(2,4)$
    is the closure of the Luna stratum where the stable representation
    of dimension vector~$\vect{e}^1_k$ splits off,
    so the action permutes the planes~$\{P_1,\ldots,P_5\}$
    in the standard way,
    and the action is effective because the planes are distinct.
    Since these planes belong to a single ruling,
    the action moreover preserves the ruling,
    so that it is realized by a subgroup~$\mathrm{S}_5\subseteq\PGL_4\subseteq\Aut(\Gr(2,4))$.
    Thus by \cref{lemma:recognition}
    the discriminant in~$\Gr(2,4)$
    corresponds to a projective frame in~$\mathbb{P}^3$.
  \item For the Segre cousin
    we use that by \cite[\S2]{MR4739832}
    the points are in general position,
    and thus provide a projective frame by \cref{lemma:uniqueness-of-configuration}.
\end{enumerate}
This allows us to identify~$P_k$ with~$P_k'$.
It remains to apply \cref{proposition:lehn}.

\begin{proof}[Proof of \cref{theorem:segre-cousin}]
  Consider the birational map~$f$ defined by
  \begin{equation}
    \label{equation:birational-morphism}
    \begin{tikzcd}
      \QM{2} \arrow["p_1"']{rd} \arrow[dashed, "f"]{rr} & & X \arrow["\rho_1"]{ld} \\
      & \Gr(2,4).
    \end{tikzcd}
  \end{equation}
  Then~$f$ is an isomorphism in codimension one,
  because the discriminant is a surface,
  and the morphisms~$p_1$ and~$\rho_1$ are semismall,
  so that~$f$ is a local isomorphism at the generic point of
  the exceptional divisor over~$P_k$
  (whilst it evidently is a local isomorphism outside the discriminant).
  By Hartogs' theorem,
  the smoothness of~$\QM{2}$, resp.~$X$,
  and the fact that both are Fano,
  we obtain that~$f^*(-\mathrm{K}_X)=-\mathrm{K}_{\QM{2}}$,
  so that \cref{proposition:lehn} applies.
\end{proof}

\begin{remark}
  \label{remark:intrinsic-rigidity}
  In \cite[Proposition~8.1]{MR4739832}
  the infinitesimal rigidity of the Segre cousin is established
  through a cohomological computation specific to this fourfold,
  combining the prehomogeneity of the space~$V_4^\vee\otimes\wedge^2V_5^\vee$
  with Koszul complex and Borel--Weil--Bott arguments
  on~$\Gr(2,V_4)\times\Gr(3,V_5)$.
  \Cref{theorem:segre-cousin} explains this rigidity more intrinsically:
  the variety~$\QM{2}\cong X$ is a quiver moduli space
  satisfying the assumptions of \cref{theorem:fano-quiver-moduli},
  and \emph{every} such quiver moduli space is infinitesimally rigid
  by \cite[Corollary~D]{MR4954467}.
  Likewise, the vanishing of global vector fields from \cref{lemma:no-vector-fields}
  is consistent with the finiteness of~$\Aut(X)\cong\mathrm{S}_5$
  established in \cite[Corollary~5.4]{MR4739832},
  with the symmetric group realized on~$\QM{2}$
  by permuting the vertices~$i_1,\ldots,i_5$ of the subspace quiver.
\end{remark}

\subsection{Second projection}
\label{subsection:QM2-second-projection}
The following analysis gives the quiver perspective
on the algebro-geometric description provided by Manivel
in \cite[\S4.2]{MR4739832}.
We consider the projection
\begin{equation}
  p_2\colon\QM{2}\to\modulispace[\overline{\Theta}_2\semistable]{Q^{(6)},\vect{d}}.
\end{equation}
We determine its fibres and the singularities of the target.

The proper non-zero dimension vectors~$\vect{e}\leq\vect{d}$ for which~$\overline{\Theta}_2(\vect{e})=0$ holds are the following:
\begin{equation}
  \vect{e}^4_{kl}=\vect{i}_k+\vect{i}_l+\vect{j},\, 1\leq k<l\leq 5,\; \vect{e}^5_k=\vect{i}_k+2\vect{i}_6+\vect{j},\, k=1,\ldots,5,
\end{equation}
as well as~$\vect{e}^4_{kl}+\vect{e}^4_{pq}$ for~$\{k,l\},\{p,q\}$ disjoint,
and~$\vect{e}^4_{kl}+\vect{e}^5_p$ for~$p\not=k,l$.
Indeed, writing~$\vect{e}$ by the triple~$(f,e,d)$,
where~$f$ is the number of vertices among~$i_1,\ldots,i_5$ occurring in~$\vect{e}$,
the equation~$\overline{\Theta}_2(\vect{e})=0$ becomes~$2f+e=4d$.
The only proper non-zero solutions are~$(2,0,1)$,~$(1,2,1)$,~$(4,0,2)$ and~$(3,2,2)$,
which give exactly the list above.
There are no~$\overline{\Theta}_2$-semistable representations of dimension vector~$\vect{e}^5_k$,
but there are stable representations of dimension vector~$\vect{e}^4_{kl}$ and~$\vect{e}^4_{kl}+\vect{e}^5_p$ for~$p\not=k,l$,
as a quick inspection of general representations of the corresponding dimension vectors shows
(see also~\cref{lemma:stable-general-position}).
Thus the only decomposition types besides~$(\vect{d}^1)$ are
\begin{equation}
  \xi= ((\vect{e}^4_{kl})^1,(\vect{e}^4_{pq}+\vect{e}^5_r)^1)
\end{equation}
for~$\{1,\ldots,5\}=\{k,l,p,q,r\}$.
The second summand depends only on the complementary triple~$\{p,q,r\}$,
so there are ten such types.

\begin{lemma}
  \label{lemma:fiber-type-qm-2-second}
  For~$V\in S_\xi\subset\modulispace[\overline{\Theta}_2\semistable]{Q^{(6)},\vect{d}}$,
  the fibre~$p_2^{-1}(V)$ is isomorphic to~$\mathbb{P}^2$.
\end{lemma}

\begin{proof}
  The corresponding local quiver
  has~$\vect{d}_\xi=\vect{i}_1+\vect{i}_2$ and~$\Theta_\xi=-\vect{i}_1^*+\vect{i}_2^*$,
  and a representation of this quiver of dimension vector~$\vect{d}_\xi$ is given by two scalars~$a_1,a_2$ on the arrows from~$1$ to~$2$
  and three scalars~$b_1,b_2,b_3$ on the arrows from~$2$ to~$1$.
  Nilpotency means that every oriented cycle acts by zero,
  hence~$a_ib_j=0$ for all~$i=1,2$ and~$j=1,2,3$.
  A nilpotent representation is~$\Theta_\xi$-semistable if and only if~$(b_1,b_2,b_3)\neq(0,0,0)$:
  indeed, the only potentially destabilising subrepresentation is the copy of~$\mathbb{C}$ at the vertex~$2$,
  of dimension vector~$\vect{i}_2$, and it exists precisely when all arrows from~$2$ to~$1$ vanish.
  On the semistable nilpotent locus we therefore have~$a_1=a_2=0$ and~$(b_1,b_2,b_3)\neq(0,0,0)$,
  so taking the quotient yields~$\mathbb{P}^2$.
\end{proof}

\begin{lemma}
  \label{lemma:singularity-qm-2-second}
  At every point of a proper Luna stratum~$S_\xi$,
  the singularity of~$\modulispace[\overline{\Theta}_2\semistable]{Q^{(6)},\vect{d}}$
  is \'etale equivalent to the vertex of the affine cone over the Segre embedding of~$\mathbb{P}^1\times\mathbb{P}^2$.
\end{lemma}

\begin{proof}
  For the trivial stability parameter every representation is semistable.
  After dividing by diagonal scalars,
  the effective base change group is~$\Gm$.
  Writing its parameter as~$\lambda=g_2/g_1$,
  the arrows from~$1$ to~$2$ have weight~$1$,
  while the arrows from~$2$ to~$1$ have weight~$-1$.
  Hence the invariant ring is generated by the six products~$a_ib_j$,
  which are the entries of a~$2\times 3$ matrix of rank at most~$1$.
  Their only relations are the~$2\times 2$ minors,
  so~$\modulispace[0\semistable]{Q_\xi,\vect{d}_\xi}$ is
  the affine cone over the Segre embedding of~$\mathbb{P}^1\times\mathbb{P}^2$.
  The claim follows from~\cref{proposition:local-quiver}.
\end{proof}

\begin{theorem}
  \label{theorem:QM2-resolution}
  The map~$p_2\colon\QM{2}\rightarrow \modulispace[\overline{\Theta}_2\semistable]{Q^{(6)},\vect{d}}$
  is a semismall resolution of singularities,
  with fibre~$\mathbb{P}^2$ over each of the ten isolated singularities.
\end{theorem}

\begin{proof}
  Over the open Luna stratum~$(\vect{d}^1)$,
  the fibre is a point.
  The proper Luna strata are the ten points indexed by the pairs~$\{k,l\}$:
  they are indeed points, both stable summands being rigid,
  as we have that~$\langle\vect{e}^4_{kl},\vect{e}^4_{kl}\rangle=\langle\vect{e}^4_{pq}+\vect{e}^5_r,\vect{e}^4_{pq}+\vect{e}^5_r\rangle=1$.
  The corresponding fibres are projective planes by~\cref{lemma:fiber-type-qm-2-second}.
  By~\cref{lemma:singularity-qm-2-second},
  these are precisely the singular points of the target.
  Since~$\QM{2}$ is smooth and~$0+2\cdot2=4$,
  the map is a semismall resolution.
\end{proof}

This is the quiver counterpart of \cite[Proposition~4.2]{MR4739832}:
under the identification of~$\QM{2}$ with the Segre cousin from \cref{theorem:segre-cousin},
the morphism~$p_2$ corresponds to the small resolution~$\rho_2\colon X\to\Gr(3,V_5)$
recalled in \cref{subsection:segre-cousin}.

\section{The variety \QM{3}: a Segre cousin once-removed}
\label{section:QM3}
We are still considering the quiver~$Q^{(6)}$,
but now with dimension vector~$\vect{d}=(1^4,2^2;3)$.
This moduli space has the same Betti numbers~$(1,6,17,6,1)$ as the previous one,
but it is not isomorphic to it: the Chern numbers in \cref{table:overview} differ.
It does, however, have a similar structure, which we will now analyze.

\subsection{Projections to the wall}
\label{subsection:QM3-projections}
Besides the canonical stability parameter
\begin{equation}
  \Theta=3\sum_{k=1}^6\vect{i}_k^*-8\vect{j}^*,
\end{equation}
we consider the stability parameters
\begin{equation}
  \begin{aligned}
    \overline{\Theta}_1&=2\sum_{k=1}^4\vect{i}_k^*+\vect{i}_5^*+\vect{i}_6^*-4\vect{j}^* \\
    \overline{\Theta}_2&=\sum_{k=1}^4\vect{i}_k^*+2\vect{i}_5^*+2\vect{i}_6^*-4\vect{j}^*.
  \end{aligned}
\end{equation}
As before, we have the following for these choices.
Instead of the proof given below,
one can also appeal to \cite{quivertools},
which implements \cite{MR5007902}.
\begin{lemma}
  The stability parameter~$\Theta$ does not belong to a proper wall,
  whilst~$\overline{\Theta}_1$ and~$\overline{\Theta}_2$ belong to the closure of the chamber containing~$\Theta$.
\end{lemma}

\begin{proof}
  Consider a non-zero proper dimension vector~$\vect{e}\leq\vect{d}$ given as
  \begin{equation}
    \vect{e}=\sum_{k\in K}\vect{i}_k+e'\vect{i}_5+e''\vect{i}_6+d\vect{j}
  \end{equation}
  for a subset~$K\subset\{1,2,3,4\}$ of cardinality~$f$ such that
  \begin{equation}
    (0,0,0,0)\lneqq(f,e',e'',d)\lneqq(4,2,2,3).
  \end{equation}
  Assume that~$\Theta(\vect{e})\leq 0$, thus~$3(f+e'+e'')\leq 8d$,
  so we have strict inequality,
  proving that~$\Theta$ does not belong to a proper wall.

  If~$\overline{\Theta}_1(\vect{e})>0$,
  that is,~$2f+e'+e''>4d$, then~$(f,e',e'',d)\in\{(4,1,0,2),(4,0,1,2)\}$.
  So assume that~$\vect{e}=\sum_{k=1}^4\vect{i}_k+\vect{i}_5+2\vect{j}$,
  and thus~$\vect{d}-\vect{e}=\vect{i}_5+2\vect{i}_6+\vect{j}$.
  A representation with this dimension vector admits
  a factor of dimension vector~$\vect{f}'=\vect{i}_5+\vect{i}_6+\vect{j}$,
  and~$\Theta(\vect{f}')=-2$.
  By~\cref{lemma:projection},
  we see that~$\overline{\Theta}_1$ belongs to the closure of the chamber containing~$\Theta$.

  Similarly we proceed for~$\overline{\Theta}_2$.
  If~$3(f+e'+e'')\leq 8d$, but~$f+2(e'+e'')>4d$, then~$(f,e',e'',d)=(1,2,2,2)$ by direct inspection.
  This then shows that every representation of dimension vector~$\vect{i}_k+2\vect{i}_5+2\vect{i}_6+2\vect{j}$
  contains a subrepresentation of dimension vector~$\vect{e}'=\vect{i}_k+\vect{i}_5+\vect{i}_6+\vect{j}$,
  and~$\Theta(\vect{e}')>0$.
  Again by~\cref{lemma:projection},
  we see that~$\overline{\Theta}_2$ belongs to the closure of the chamber containing~$\Theta$.
\end{proof}

\subsection{First projection}
\label{subsection:QM3-first-projection}
We describe the Luna strata for~$\overline{\Theta}_1$.
The only proper non-zero dimension vectors~$\vect{e}\leq\vect{d}$
containing~$\vect{j}$ with multiplicity one,
satisfying~$\overline{\Theta}_1(\vect{e})=0$,
and admitting a~$\overline{\Theta}_1$-stable representation~(cf.~\cref{lemma:stable-general-position}) are the
\begin{equation}
  \vect{e}^1_k=\vect{i}_k+\vect{i}_5+\vect{i}_6+\vect{j},\, k=1,\ldots,4,\; \vect{e}^2_{kl}=\vect{i}_k+\vect{i}_l+\vect{j},\, 1\leq k<l\leq 4.
\end{equation}
Then one can easily list the Luna strata besides~$(\vect{d}^1)$ as
\begin{equation}
  ((\vect{e}^1_k)^1,(\vect{d}-\vect{e}^1_k)^1)
\end{equation}
for~$k=1,2,3,4$ and
\begin{equation}
  ((\vect{e}^1_k)^1,(\vect{e}^1_l)^1,(\vect{e}^2_{pq})^1)
\end{equation}
for~$\{k,l,p,q\}=\{1,2,3,4\}$.

In the first case, the local quiver is the \emph{opposite} of the one underlying \eqref{quiver122},
with~$\vect{d}_\xi=\vect{i}_1+\vect{i}_2$ and~$\Theta_\xi=\vect{i}_1^*-\vect{i}_2^*$:
indeed, here we have~$\Theta(\vect{e}^1_k)=1$ and~$\Theta(\vect{d}-\vect{e}^1_k)=-1$,
opposite in sign to the situation in \cref{subsection:QM2-first-projection}.
Passing to the opposite quiver does not change the moduli spaces by \cref{lemma:linear-duality},
so by \cref{lemma:fiber-type-qm-2-1,lemma:stratum-qm-2-1},
the corresponding moduli space of semistable nilpotent representations is again~$\mathbb{P}^1$,
and the moduli space for the trivial stability parameter is smooth,
thus there are no singularities in the corresponding Luna strata.

In the second case, the local quiver is the opposite of the one underlying \eqref{quiver1212},
with~$\vect{d}_\xi=\vect{i}_1+\vect{i}_2+\vect{i}_3$ and~$\Theta_\xi=\vect{i}_1^*-2\vect{i}_2^*+\vect{i}_3^*$,
as~$\Theta(\vect{e}^1_k)=\Theta(\vect{e}^1_l)=1$ and~$\Theta(\vect{e}^2_{pq})=-2$.
Again combining \cref{lemma:linear-duality}
with \cref{lemma:fiber-type-qm-2-2,lemma:stratum-qm-2-2},
the corresponding moduli space is~$\mathbb{P}^1\times\mathbb{P}^1$, as before,
and the moduli space for the trivial stability parameter
is also smooth.
We have thus proved:

\begin{proposition}
  The projection~$p_1\colon\QM{3}\rightarrow \modulispace[\overline{\Theta}_1\semistable]{Q^{(6)},\vect{d}}$ is birational with smooth target,
  with fibre over four punctured surfaces isomorphic to~$\mathbb{P}^1$,
  and fibre over six points isomorphic to~$\mathbb{P}^1\times\mathbb{P}^1$.
  In particular, it is semismall.
\end{proposition}

We consider a representation of~$Q^{(6)}$ of dimension vector~$\vect{d}$ as a tuple
\begin{equation}
  (v_1,\ldots,v_4,A,B)
\end{equation}
of vectors~$v_i\in\mathbb{C}^3$ and two~$3\times 2$-matrices~$A,B$, which we represent as the~$3\times 8$-matrix
\begin{equation}
  M=[v_1\mid v_2\mid v_3\mid v_4\mid A\mid B].
\end{equation}
We consider the semi-invariants
\begin{equation}
  D_i=\det[v_p\mid v_q\mid v_r]
\end{equation}
for~$i=1,\ldots,4$ and~$\{1,2,3,4\}\setminus\{i\}=\{p<q<r\}$, as well as
\begin{equation}
  E_i^A=\det[v_i\mid A],\; E_i^B=\det[v_i\mid B]
\end{equation}
for~$i=1,\ldots,4$. We then find semi-invariants
\begin{equation}
  S_{kl}=D_kE_k^AD_lE_l^B
\end{equation}
for~$k,l=1,\ldots,4$ of weight~$\overline{\Theta}_1$.

\begin{lemma}
  The map from~$\repspace[\overline{\Theta}_1\semistable]{Q^{(6)},\vect{d}}$ to~$\mathbb{P}^{15}$
  given by these coordinates is well-defined.
\end{lemma}

\begin{proof}
  Suppose, to the contrary, that all~$S_{kl}$ evaluate to zero on a matrix~$M$.
  Without loss of generality, we can assume that all~$D_kE_k^A$ evaluate to zero.
  For a~$\overline{\Theta}_1$-semistable representation,
  all arrows are represented by injections,
  since the kernel of any arrow would define a subrepresentation supported at a source vertex,
  on which~$\overline{\Theta}_1$ is positive.
  Thus all vectors~$v_i$ are non-zero,
  and the matrices~$A,B$ have rank~$2$;
  let~$H_A,H_B\subset \mathbb{C}^3$ be the span of the columns of~$A$ and~$B$, respectively.

  The vectors~$v_1,\ldots,v_4$ span~$\mathbb{C}^3$:
  otherwise their span is contained in a hyperplane~$H'$, which intersects~$H_A$ in a non-zero vector~$v$.
  This induces a subrepresentation of dimension vector~$\sum_{k=1}^4\vect{i}_k+\vect{i}_5+2\vect{j}$, a contradiction.
  Thus there exists a three-element subset~$K\subset\{1,\ldots,4\}$ such that~$(v_k)_{k\in K}$ forms a basis of~$\mathbb{C}^3$.

  By assumption,~$v_l\in H_A$ for the index~$l\not\in K$.
  At most two of the vectors~$v_1,\ldots,v_4$ are contained in~$H_A$:
  if three vectors~$v_p,v_q,v_r$ are contained in~$H_A$
  and~$v$ is a non-zero vector in the intersection of~$H_A$ and~$H_B$,
  we find a subrepresentation of dimension vector~$\vect{i}_p+\vect{i}_q+\vect{i}_r+2\vect{i}_5+\vect{i}_6+2\vect{j}$, a contradiction.
  Assume there is a vector~$v_l$ besides~$v_k$ contained in~$H_A$.
  If~$v_k$ and~$v_l$ are linearly dependent,
  we find a subrepresentation of dimension vector~$\vect{i}_k+\vect{i}_l+\vect{i}_5+\vect{j}$, a contradiction.
  Thus~$v_k$ and~$v_l$ span~$H_A$, and~$\{1,2,3,4\}=\{k,l,p,q\}$.
  Then~$v_k,v_l,v_p$ as well as~$v_k,v_l,v_q$ are linearly dependent by assumption, again a contradiction.
\end{proof}

We easily identify the subvariety of~$\mathbb{P}^{15}$ defined by the above coordinates~$S_{kl}$
as the image of the Segre embedding of~$\mathbb{P}^3\times\mathbb{P}^3$ by factoring
\begin{equation}
  S_{kl}=X_k^AX_l^B,\; X_k^A=D_kE_k^A,\; X_k^B=D_kE_k^B.
\end{equation}
The Pl\"ucker relations for the matrix~$M$ show that, additionally,
\begin{equation}
  \sum_{k=1}^4(-1)^kX_k^A=0=\sum_{k=1}^4(-1)^kX_k^B.
\end{equation}

We thus have a well-defined map~$p\colon\modulispace[\overline{\Theta}_1\semistable]{Q^{(6)},\vect{d}}\rightarrow\mathbb{P}^2\times\mathbb{P}^2$.
\begin{lemma}
  \label{lemma:QM3-target-quadrics}
  The map~$p\colon\modulispace[\overline{\Theta}_1\semistable]{Q^{(6)},\vect{d}}\rightarrow\mathbb{P}^2\times\mathbb{P}^2$
  is an isomorphism.
  It maps the closures of the four two-dimensional Luna strata
  to the four quadric surfaces
  \begin{equation}
    Q_k=\{X_k^A=0\}\times\{X_k^B=0\}\cong\mathbb{P}^1\times\mathbb{P}^1,
  \end{equation}
  and the six point strata to the six pairwise intersection points~$Q_k\cap Q_l$.
\end{lemma}

\begin{proof}
  Since the target is normal,
  it suffices by Zariski's main theorem to prove that
  the proper birational morphism~$p$ is bijective,
  which we do stratum by stratum.

  If~$(v_1,v_2,v_3,v_4,A,B)$ defines a~$\overline{\Theta}_1$-stable representation,
  every three~$v_p,v_q,v_r$ of the four vectors already form a basis:
  otherwise, let~$H$ be their span, and choose non-zero vectors~$v_A\in H\cap H_A$,~$v_B\in H\cap H_B$.
  They define a subrepresentation of dimension vector~$\vect{i}_p+\vect{i}_q+\vect{i}_r+\vect{i}_5+\vect{i}_6+2\vect{j}$,
  a contradiction to stability.
  We can thus normalize~$M$ to the form
  \begin{equation}
    [e_1\mid e_2\mid e_3\mid e_1+e_2+e_3\mid A\mid B]
  \end{equation}
  and read off the maximal minors of~$A$ and~$B$ from the~$D_kE_k^A$ and~$D_kE_k^B$,
  proving bijectivity over the open Luna stratum.
  Moreover, the image of a stable representation avoids the four quadrics~$Q_k$:
  since all~$D_k\not=0$,
  having~$X_k^A=X_k^B=0$ would mean~$v_k\in H_A\cap H_B$,
  which produces a subrepresentation of dimension vector~$\vect{e}^1_k$,
  contradicting stability as~$\overline{\Theta}_1(\vect{e}^1_k)=0$.

  Next, consider the two-dimensional Luna stratum~$((\vect{e}^1_4)^1,(\vect{d}-\vect{e}^1_4)^1)$, say.
  A polystable representation is the direct sum of
  the unique representation of dimension vector~$\vect{e}^1_4$,
  supported on a line in~$\mathbb{C}^3$,
  and a~$\overline{\Theta}_1$-stable representation of dimension vector~$\vect{d}-\vect{e}^1_4$,
  supported on a complementary plane,
  so we can normalize~$M$ to
  \begin{equation}
    M=\left[
      \begin{array}{cccccccc}0&0&0&1&1&0&1&0\\ w_1&w_2&w_3&0&0&a&0&b
    \end{array}\right]
  \end{equation}
  for~$w_1,w_2,w_3,a,b\in\mathbb{C}^2$,
  where the~$w_k$ are pairwise linearly independent:
  a proportionality between~$w_k$ and~$w_l$ would give a subrepresentation
  of dimension vector~$\vect{e}^2_{kl}$ of the second summand, contradicting its stability.
  We compute~$D_4=0$ and~$E_4^A=E_4^B=0$,
  whilst~$D_k=\pm\det[w_p\mid w_q]\not=0$ for~$k\leq 3$ and~$\{k,p,q\}=\{1,2,3\}$, so that
  \begin{equation}
    X_4^A=X_4^B=0,\quad
    X_k^A=\pm\det[w_p\mid w_q]\det[w_k\mid a],\quad
    X_k^B=\pm\det[w_p\mid w_q]\det[w_k\mid b].
  \end{equation}
  Thus the stratum maps into the quadric~$Q_4$,
  and normalizing~$(w_1,w_2,w_3)$ to~$(e_1,e_2,e_1+e_2)$,
  the remaining moduli~$([a],[b])\in\mathbb{P}^1\times\mathbb{P}^1$
  of the second summand
  are recovered from~$(X_1^A:X_2^A:X_3^A)$ and~$(X_1^B:X_2^B:X_3^B)$,
  proving injectivity on the stratum,
  and identifying the closure of its image with~$Q_4\cong\mathbb{P}^1\times\mathbb{P}^1$.
  The image omits exactly the three points~$[a]=[b]=[w_k]$,
  where the second summand becomes strictly semistable
  against the subrepresentation of dimension vector~$\vect{e}^1_k$.

  Finally, consider the six point strata~$((\vect{e}^1_k)^1,(\vect{e}^1_l)^1,(\vect{e}^2_{pq})^1)$;
  these are indeed points, all three summands being unique representations
  of their dimension vectors.
  Now~$v_p$ and~$v_q$ are proportional,
  so that~$D_k=D_l=0$,
  whilst~$D_p,D_q\not=0$
  and~$E_p^A,E_q^A,E_p^B,E_q^B\not=0$,
  as the columns of~$A$ and~$B$ span the plane
  supporting the first two summands,
  which does not contain~$v_p$ and~$v_q$.
  Thus precisely the coordinates~$X_k^A,X_l^A,X_k^B,X_l^B$ vanish,
  and the stratum maps to the intersection point~$Q_k\cap Q_l$,
  which consists of a single point
  as two distinct lines in~$\mathbb{P}^2$ meet in one point.
  These six points are pairwise distinct:
  no three of the lines~$\{X_k^A=0\}$ are concurrent,
  since the four coordinates~$X_k^A$ cannot vanish simultaneously on~$\mathbb{P}^2$.

  Summarizing, the eleven Luna strata map injectively
  to the pairwise disjoint locally closed subvarieties
  \begin{equation}
    \mathbb{P}^2\times\mathbb{P}^2\setminus\bigcup_{k=1}^4Q_k,\quad
    Q_k\setminus\bigcup_{l\not=k}(Q_k\cap Q_l),\quad
    Q_k\cap Q_l,
  \end{equation}
  which cover~$\mathbb{P}^2\times\mathbb{P}^2$.
  As~$p$ is moreover projective and dominant, hence surjective,
  it is bijective, and the claim follows.
\end{proof}

We have thus proved:
\begin{theorem}
  \label{theorem:QM3-fibration}
  The projection~$p_1\colon\QM{3}\rightarrow\mathbb{P}^2\times\mathbb{P}^2$ is a projective birational semismall morphism.
  In the target there is a configuration of four copies of~$\mathbb{P}^1\times\mathbb{P}^1$,
  namely the quadrics~$Q_1,\ldots,Q_4$ from \cref{lemma:QM3-target-quadrics},
  meeting pairwise in a single point, for a total of six points;
  over the complement of these points in the union of the quadrics
  the fibre is isomorphic to~$\mathbb{P}^1$,
  while over the six intersection points the fibre is isomorphic to~$\mathbb{P}^1\times\mathbb{P}^1$.
\end{theorem}
This is the analogue of \cref{theorem:QM2-grassmannian},
with~$\Gr(2,4)$ being replaced by~$\mathbb{P}^2\times\mathbb{P}^2$,
and the configuration of five planes
being replaced by four quadric surfaces.

\begin{remark}
  \label{remark:enrico}
  Upon seeing a draft of this paper,
  Enrico Fatighenti realized that there also exists a zero locus description of $\QM{3}$,
  despite it not appearing in any of the existing databases of Fano 4-fold zero loci.
  By encoding the blowup description from \cref{theorem:QM3-fibration}
  using vector bundles,
  he shows it is possible to realize~$\QM{3}$ as
  the zero locus of a general section of a rank-$16$ vector bundle
  on the~$20$-dimensional product
  \begin{equation}
    \mathbb{P}^2\times\mathbb{P}^2\times(\mathbb{P}^4)^4.
  \end{equation}
  Write~$\pi_1,\ldots,\pi_6$ for the projections of this product onto its six factors,
  and let~$\mathcal{Q}$ denote the rank-two tautological quotient bundles
  on~$\mathbb{P}^2$,
  defined by the Euler sequence~$0\to\mathcal{O}_{\mathbb{P}^2}(-1)\to\mathcal{O}^{\oplus3}\to\mathcal{Q}\to 0$.
  The bundle is given by
  \begin{equation}
    \mathcal{E}
    \colonequals
    \bigoplus_{j=3}^{6}
    \left(
      \pi_1^*\mathcal{Q}\otimes \pi_j^*\mathcal{O}_{\mathbb{P}^4}(1)
      \oplus
      \pi_2^*\mathcal{Q}\otimes \pi_j^*\mathcal{O}_{\mathbb{P}^4}(1)
    \right).
  \end{equation}

  It is a direct sum of four blocks, one for each~$\mathbb{P}^4$-factor:
  the summand indexed by~$j$ involves only the projections~$\pi_1$,~$\pi_2$,~$\pi_j$,
  and cuts out the blowup of~$\mathbb{P}^2\times\mathbb{P}^2$ along one of the four surfaces~$\mathbb{P}^1\times\mathbb{P}^1$ of~\cref{theorem:QM3-fibration}.
  Thus~$\QM{3}$ can be shown to be
  the fibre product over~$\mathbb{P}^2\times\mathbb{P}^2$ of these four blowups,
  and the projection~$(\pi_1,\pi_2)$ onto~$\mathbb{P}^2\times\mathbb{P}^2$
  restricts on it to the semismall morphism of~\cref{theorem:QM3-fibration}.
  Its anticanonical class is the restriction of~$-H_1-H_2+\sum_{j=3}^{6}h_j$,
  where~$H_i=\pi_i^*\mathcal{O}_{\mathbb{P}^2}(1)$ and~$h_j=\pi_j^*\mathcal{O}_{\mathbb{P}^4}(1)$
  are the hyperplane classes of the factors.
  Since this is not the restriction of an ample class from the ambient product
  (a Fano zero locus for which the anticanonical bundle is the restriction is called strongly Fano),
  the description falls outside the existing databases.
\end{remark}

\begin{remark}
  \label{remark:QM3-second-projection}
  As for~$\QM{2}$,
  one can also consider the projection~$p_2\colon\QM{3}\to\modulispace[\overline{\Theta}_2\semistable]{Q^{(6)},\vect{d}}$
  for the second stability parameter.
  The analogy with \cref{subsection:QM2-second-projection} is imperfect though:
  there are no~$\overline{\Theta}_2$-stable representations of dimension vector~$\vect{d}$,
  because a non-zero vector~$v\in H_A\cap H_B$ induces
  a subrepresentation of dimension vector~$\vect{i}_5+\vect{i}_6+\vect{j}$,
  on which~$\overline{\Theta}_2$ vanishes.
  In fact, a general~$\overline{\Theta}_2$-polystable representation is the direct sum of
  the unique stable representation of dimension vector~$\vect{i}_5+\vect{i}_6+\vect{j}$
  and a stable representation of dimension vector~$\sum_{k=1}^6\vect{i}_k+2\vect{j}$,
  so that~$p_2$ is a fibration onto a threefold
  instead of a birational morphism.
  One can identify this threefold with
  the toric del Pezzo threefold given by the intersection of two quadrics with six ordinary double points.
\end{remark}

\section{The variety \QM{4}: moduli of points and the Fano model of \texorpdfstring{$\Bl_6\mathbb{P}^4$}{Bl6 P4}}
\label{section:QM4}

This moduli space has Betti numbers~$(1,7,22,7,1)$.
We first exhibit the morphism to the Segre cubic announced in \cref{theorem:main},
and then identify~$\QM{4}$ with the Fano model of~$\Bl_6\mathbb{P}^4$
in \cref{subsection:fano-model}.

\subsection{Framed structure}
By~\cref{proposition:subspace-framed}, the moduli space~$\QM{4}$ is isomorphic to the framed moduli space of~$\modulispace{1^6;2}$ with framing datum~$\vect{j}$.
The latter moduli space is nothing other than the Segre cubic~$S$ by \cite{MR2450346},
that is, the GIT quotient of six ordered points on~$\mathbb{P}^1$ with symmetric linearization.
It has eleven Luna strata; the open one is given by the decomposition type~$((1^6;2)^1)$,
and the ten others correspond to the ten singular points of~$S$.
Namely, for any unordered decomposition~$\{1,2,3,4,5,6\}=K\cup\overline{K}$ into disjoint three-element subsets,
we have the decomposition type~$(\vect{d}_K^1,\vect{d}_{\overline{K}}^1)$,
where~$\vect{d}_{\{p,q,r\}}=\vect{i}_p+\vect{i}_q+\vect{i}_r+\vect{j}$.

This allows us to determine the fibres of the projection~$p\colon\QM{4}\rightarrow S$.
The general fibre is~$\mathbb{P}^1$.
The fibre over each of the ten singular points is given as the moduli space of stable nilpotent representations of the framing of the local quiver,
for dimension vector~$\vect{i}_0+\vect{i}_1+\vect{i}_2$ and stability~$2\vect{i}_0^*-\vect{i}_1^*-\vect{i}_2^*$. Such a representation is given by
\begin{equation}
  \begin{tikzpicture}[node distance = 2.5cm]
    \node (a) {$\mathbb{C}$};
    \node (b) [right of = a] {$\mathbb{C}$};
    \node (c) at ($(a)!0.5!(b) + (0,1.1cm)$) {$\mathbb{C}$};
    \draw[->] (c) edge node [left] {$x$} (a);
    \draw[->] (c) edge node [right] {$y$} (b);
    \draw[->, bend left = 45] (a) edge node [above] {$a$} (b);
    \draw[->, bend left = 15] (a) edge node [midway, fill=white] {$b$} (b);
    \draw[->, bend left = 15] (b) edge node [midway, fill=white] {$c$} (a);
    \draw[->, bend left = 45] (b) edge node [below] {$d$} (a);
  \end{tikzpicture}
\end{equation}
for scalars~$x,y,a,b,c,d$ such that~$ac=ad=bc=bd=0$, and such that the vectors~$(x,cy,dy)$ and~$(y,ax,bx)$ are non-zero. This provides us with an embedding of the moduli space into~$\mathbb{P}^2\times\mathbb{P}^2$ with coordinates~$((x_0:x_1:x_2),(y_0:y_1:y_2))=((x:cy:dy),(y:ax:bx))$, whose defining equations~$x_iy_j=0$ for~$i,j=1,2$ are immediate from~$ac=ad=bc=bd=0$. This defines the union~$\mathbb{P}^2\vee\mathbb{P}^2$ of two copies of~$\mathbb{P}^2$ glued at a single point.

This description is compatible with the birational geometry of symmetric quotients of point configurations on~$\mathbb{P}^1$ studied by Bolognesi and Massarenti \cite{MR4243655},
to which we return in \cref{subsection:fano-model};
our point here is that quiver-moduli techniques produce this fibre picture directly from local quiver calculations.

\begin{theorem}
  \label{theorem:QM4-segre-cubic}
  The Fano fourfold~$\QM{4}$ admits a morphism to the Segre cubic,
  with general fibre~$\mathbb{P}^1$,
  and fibre~$\mathbb{P}^2\vee\mathbb{P}^2$ over each of the ten singular points.
\end{theorem}
In particular, the morphism is not flat:
the dimension of the fibres jumps from~1 to~2 over the singular points.

\subsection{Identification with the Fano model of \texorpdfstring{$\Bl_6\mathbb{P}^4$}{Bl6P4}}
\label{subsection:fano-model}
By \cref{theorem:main},
the moduli space~$\QM{4}$ is the GIT quotient~$((\mathbb{P}^1)^7)_{\rm st}/\!/\PGL_2$,
where the stability condition \eqref{equation:stability}
induced by the canonical stability parameter
is the classical stability condition for the symmetric linearization:
a configuration of seven ordered points on~$\mathbb{P}^1$ is (semi)stable
if and only if at most three of the points coincide,
there being no strictly semistable configurations because the number of points is odd.
Thus~$\QM{4}$ is the classical GIT moduli space of seven points on the line,
denoted~$\Sigma_4$ in \cite{MR4243655},
in the same way that the Segre cubic~$S$ is
the moduli space~$\Sigma_3$ of six points on the line
\cite[Example~1.4]{MR4243655}.

The fourfold~$\Bl_6\mathbb{P}^4$ is itself not Fano:
the strict transform~$\ell$ of a line through two of the six points
satisfies~$-\mathrm{K}\cdot\ell=-1$.
Flipping the fifteen such lines produces a small modification
on which the anticanonical divisor is ample,
called the \emph{Fano model} of~$\Bl_6\mathbb{P}^4$:
it is a smooth Fano fourfold of Picard rank~7,
which is neither toric nor a product,
see \cite[Example~25]{MR4622141}.
The identification announced in \cref{theorem:main}
is then the following result of Bolognesi--Massarenti \cite[\S2.35]{MR4243655}.

\begin{proposition}
  \label{proposition:QM4-fano-model}
  The quiver moduli space~$\QM{4}$ is isomorphic to the Fano model of~$\Bl_6\mathbb{P}^4$.
\end{proposition}

\begin{remark}
  Under the identification~$\QM{4}\cong\Sigma_4$,
  the morphism~$p\colon\QM{4}\to S$ from \cref{theorem:QM4-segre-cubic}
  corresponds to the forgetful morphism~$\Sigma_4\to\Sigma_3$,
  and the description of its fibres agrees with \cite[\S2.40--2.42]{MR4243655}.
  Similarly,
  the Hilbert series of~$\Sigma_4$ computed in \cite[Corollary~3.5]{MR4243655},
  together with the identification~$\omega_{\Sigma_4}\cong\mathcal{O}_{\Sigma_4}(-1)$
  of \cite[Lemma~2.28]{MR4243655},
  gives~$\deg\Sigma_4=154$ and~$\mathrm{h}^0(\Sigma_4,\mathcal{O}_{\Sigma_4}(1))=36$,
  independently confirming the values of~$\mathrm{c}_1^4$ and~$\mathrm{h}^0(\omega^\vee)$
  in \cref{table:overview}.
  Moreover,~$\Aut(\QM{4})\cong\mathrm{S}_7$ by \cite[Theorem~4.7]{MR4243655}:
  all automorphisms are induced by permutations of the seven points,
  i.e., of the vertices~$i_1,\ldots,i_7$ of the subspace quiver~$Q^{(7)}$,
  consistently with \cref{lemma:no-vector-fields}.
\end{remark}

\renewcommand*{\bibfont}{\normalfont\small}
\setlength{\bibitemsep}{1pt}
\printbibliography

\emph{Pieter Belmans}, \url{p.belmans@uu.nl} \\
Mathematical Institute, Utrecht University, Budapestlaan 6, 3584CD Utrecht, Netherlands

\emph{Markus Reineke}, \url{markus.reineke@ruhr-uni-bochum.de} \\
Faculty of Mathematics, Ruhr-Universit\"at Bochum, Universit\"atsstra\ss e 150, 44780 Bochum, Germany

\end{document}